\documentclass[a4paper,12pt,twoside]{article}
\pdfoutput=1 
\usepackage[
  top=2cm,
  bottom=4cm,
  heightrounded,
  left=2.98cm,
  right=2.98cm,
  bindingoffset=0.0cm
]{geometry}

\usepackage[utf8]{inputenc}
\usepackage[T1]{fontenc}
\usepackage{lmodern}
\usepackage[british]{babel}
\usepackage{microtype}
\usepackage[skip=0.5\baselineskip plus 2.8pt]{parskip}

\usepackage[sectionbib]{natbib}
\setcitestyle{numbers,comma,square}

\usepackage[pdfencoding=auto]{hyperref}
\hypersetup{colorlinks=true, linkcolor=blue!30!black, citecolor=green!15!black, urlcolor=red!35!black}

\usepackage{amsfonts,amssymb,bbm,stmaryrd}
\usepackage[scr=boondoxo,scrscaled=1.05]{mathalfa} %
\usepackage{amsmath}
\usepackage{amsthm}

\usepackage{mathtools}
\usepackage[noabbrev,capitalise]{cleveref}
\usepackage{centernot,float,subcaption,url,array,booktabs,colortbl}

\usepackage[separate-uncertainty=true,multi-part-units=single]{siunitx}

\usepackage[shortlabels]{enumitem}
\setlist[enumerate,1]{label={\roman*)}} %

\usepackage{tikz}
\usetikzlibrary{matrix,positioning}
\usetikzlibrary{calc}
\usetikzlibrary{arrows}
\tikzset{> =stealth}

\pgfdeclarelayer{bg}
\pgfdeclarelayer{over}
\pgfsetlayers{bg,main,over}

\newcommand{\addQEDstyle}[2]{\AtBeginEnvironment{#1}{\pushQED{\qed}\renewcommand{\qedsymbol}{#2}}\AtEndEnvironment{#1}{\popQED}}

\theoremstyle{plain}
\newtheorem{theorem}{Theorem}[section]
\newtheorem{lemma}[theorem]{Lemma}
\newtheorem{proposition}[theorem]{Proposition}
\newtheorem{corollary}[theorem]{Corollary}

\theoremstyle{definition}
\newtheorem{definition}[theorem]{Definition}

\theoremstyle{remark}
\newtheorem{remark}[theorem]{Remark}

\newtheorem{example}[theorem]{Example}
\addQEDstyle{example}{$\triangle$}

\renewcommand{\epsilon}{\varepsilon}
\renewcommand{\phi}{\varphi}

\mathchardef\mhyphen="2D

\renewcommand{\land}{\mathrel{\wedge}}
\renewcommand{\lor}{\mathrel{\vee}}

\newcommand{\N}{\mathbb{N}}

\newcommand{\Q}{\mathbb{Q}}
\newcommand{\R}{\mathbb{R}}
\newcommand{\C}{\mathbb{C}}

\newcommand{\lowerReals}{\overrightarrow{\R}}
\newcommand{\upperReals}{\overleftarrow{\R}}

\newcommand{\id}{\mathrm{id}}

\newcommand{\Set}{\mathbf{Set}}
\newcommand{\Pos}{\mathbf{Pos}}
\newcommand{\Sup}{\mathbf{Sup}}
\newcommand{\Inf}{\mathbf{Inf}}

\newcommand{\Frm}{\mathbf{Frm}}
\newcommand{\Loc}{\mathbf{Loc}}

\newcommand{\Top}{\mathbf{Top}}
\newcommand*{\SLat}[1]{{#1}\mhyphen\mathbf{SLat}}

\newcommand{\op}{{^\mathrm{\hspace{0.5pt}op}}}
\newcommand{\comp}{{^\mathsf{c}}}
\newcommand{\heyting}{\mathbin{\Rightarrow}}

\DeclareMathOperator{\Hom}{Hom}
\renewcommand{\O}{\mathcal{O}}
\newcommand{\Sh}{\mathrm{Sh}}
\newcommand{\D}{\mathcal{D}}
\newcommand{\Pfin}{\mathcal{P}_\mathrm{fin}}
\newcommand{\Sloc}{\mathcal{S}}
\newcommand{\Idl}{\mathscr{I}}

\newcommand{\llround}{\ensuremath{(\!(}}
\newcommand{\rrround}{\ensuremath{)\!)}}

\begin{document}

\title{Pointfree topology and constructive mathematics}
\author{Graham Manuell\thanks{Department of Mathematical Sciences, Stellenbosch University, South Africa}}
\date{}

\maketitle

Constructive mathematics has many advantages over its classical counterpart: vague promises of existence are replaced with concrete examples and theorems can be unwound to obtain algorithms. Even better, constructive results can be interpreted in any topos, giving many different results for the price of one.

Unfortunately, there is also a cost. We are told that we must give up many familiar results: not only Tychonoff's theorem and the Hahn--Banach theorem, but the intermediate value theorem, the extreme value theorem and the fundamental theorem of algebra.
We can sympathise with David Hilbert who asserted that asking a mathematician to avoid excluded middle is akin to ``proscribing the telescope to the astronomer or to the boxer the use of his fists''.
However, despite what you might have heard, all the results listed above are still true constructively!
They simply need to be formulated in the appropriate way.

The cost is still nonzero, since most mathematicians lack familiarity with constructive techniques, and it can take a little time to understand what is permitted and what is not. Nonetheless, the obstacles are not insurmountable. I hope that these notes will provide some assistance, as well as some motivation for further study.

The solution to the problems above lies in a reformulation of topology. This is perhaps the area where the apparent problems of constructive mathematics are most acute. It is true that point-set topology is very badly behaved constructively, but the problems disappear if we instead use the \emph{pointfree} approach to topology.
Pointfree topology studies spaces through their lattices of opens without reference to a predefined underlying set of points and can also be understood from a logical perspective.
While such an algebraic/logical approach has some advantages even in classical mathematics, it really shines in its smooth generalisation to the constructive setting.

But as we will see, pointfree topology does more than just allow for the constructive development of topology; the topological perspective also sheds light on aspects of constructive mathematics that might at first appear to have little to do with topology. From this point of view even a version of the axiom of choice can be understood as a constructive theorem.

This chapter will provide a gentle introduction to the main aspects of constructive topology and some of its applications.

\section{Motivation and basic concepts}

\subsection{Topology}

In order to reformulate topology it is worthwhile to first understand what topology is actually studying.
Let us forget any preconceived notions might have for a moment and imagine a simple situation.

Suppose we have some physical object that we want to measure. For instance, we might want to determine the length $\ell$ of a steel rod.
In the real world we cannot measure the length of the rod to infinite precision. Perhaps we use a ruler to find that it has a length of $\SI{10.1 \pm 0.3}{cm}$.
If we wish to be more precise we could use a ruler with finer gradations and measure the rod to be $\SI{10.0 \pm 0.1}{cm}$ long.
We might suspect that the rod is exactly $\SI{10}{cm}$ in length, but note that only being able to measure to finite precision we will never be able to be certain that this is true.

On the other hand, suppose we are told that the rod is strictly less than $\SI{10}{cm}$. This is something we will be able to verify for ourselves using arbitrarily, but not infinitely, precise physical measurements.
If $\ell < \SI{10}{cm}$ then there is some $\epsilon > 0$ such that $\ell + \epsilon < \SI{10}{cm}$ and so if our measurement error is less than $\epsilon$ we can check that this is indeed the case. We do not know $\epsilon$ beforehand, but if we keep on taking more and more precise measurements, at some point we will be able to tell.

Do note, however, that if it is not true that $\ell < \SI{10}{cm}$, then we might continue to make finer and finer measurements without ever being certain we have been misled (since perhaps $\ell = \SI{10}{cm}$ exactly). We say that the proposition $\ell < \SI{10}{cm}$ is \emph{verifiable}, but the proposition $\ell \ge \SI{10}{cm}$ is not.

There is a space $\R^+$ of all possible lengths for the rod, and the verifiable properties correspond precisely to the open subsets of $\R^+$.
We will now use verifiability as our \emph{definition} of open sets, even in situations that are less familiar.
For further discussion of this approach see Chapter 2 of \citet*{vickers1989book}.

The power of this approach is that it immediately motivates the definition of a topological space.
Notice that verifiable properties satisfy the following conditions.
\begin{enumerate}
  \item The property that is always true is verifiable.
  \item If $U$ and $V$ are verifiable properties, then so is $U \wedge V$.
  \item If $\mathcal{U}$ is a set of verifiable properties, then $\bigvee \mathcal{U}$ is verifiable.
\end{enumerate}
The first condition holds since if the property is always true, then there is nothing to check. For the second, we can simply check $U$ and $V$ in turn in order to verify their conjunction. Observe there is no apparent way that we could verify an \emph{infinite} conjunction of such properties, since this process would take infinitely long. On the other hand, even infinite disjunctions of verifiable properties are verifiable, since if $\bigvee \mathcal{U}$ holds, so must some $U \in \mathcal{U}$ and this $U$ can then simply be verified in finite time. Note that the empty disjunction $\bot$ is vacuously verifiable, since we only need to be able to verify propositions when they are true.

We obtain the definition of a topological space by starting with a set $X$ of things that might satisfy some properties and identifying verifiable properties on $X$ with the subsets of things that satisfy them.
Furthermore, if $f\colon X \to Y$ is a physically realisable function then for an open $U$ of $Y$ we can verify that $x \in f^{-1}(U)$ by checking that $f(x) \in U$. Thus, $f^{-1}(U)$ should be an open of $X$. This motivates the definition of a continuous map.

\subsubsection{Frames and geometric logic}

Instead of identifying verifiable properties with sets of points we can also treat them more abstractly. The logic of verifiable propositions is called \emph{geometric logic}.
A formula in geometric logic can be built up from `basic propositions' using finite conjunctions or arbitrary disjunctions.
There is no implication connective, since implications are not generally verifiable. However, we can describe top-level implications between formulae by \emph{sequents} $\phi \vdash \psi$, which are interpreted as meaning that $\psi$ holds whenever $\phi$ does.

A \emph{geometric theory} is specified by giving by a set of basic propositions (such as $\ell < \SI{10}{cm}$ in the example above) and a set of axioms (such as $\ell < \SI{10}{cm} \vdash \ell < \SI{11}{cm}$). From these we can derive additional sequents using intuitive logical rules. (For example, if we have $\phi_\alpha \vdash \psi$ for all $\alpha \in I$, then we can conclude that $\bigvee_{\alpha \in I} \phi_\alpha \vdash \psi$. For more details, see Section 4 of \citet*{pittsLogic}, but ignore the rules concerning implication and negation, and change the finite disjunctions to infinite ones.)

Just as classical propositional logic can be studied algebraically using Boolean algebras, geometric logic can be studied using frames.
\begin{definition}
 A \emph{frame} is a poset with finite meets\footnote{It suffices to ask for nullary meets (i.e.\ the top element $1$) and binary meets.} (or infima) and arbitrary joins\footnote{Any poset admitting arbitrary joins also has arbitrary meets by a standard argument and is hence a complete lattice. However, the infinite meets are merely incidental to the structure of a frame. They do not have a logical interpretation and will not be preserved by frame homomorphisms.} (or suprema) satisfying the distributivity condition $a \wedge \bigvee_{\alpha \in I} b_\alpha = \bigvee_{\alpha \in I} a \wedge b_\alpha$,
 where $\wedge$ denotes a binary meet and $\bigvee$ denotes an arbitrary join.
\end{definition}
Elements of a frame correspond to (classes of provably equivalent) geometric formulae. Meets and joins in the frame correspond to the conjunctions and disjunctions in geometric logic and the distributivity condition ensures that the conjunctions and disjunctions interact as we expect them to. A sequent is simply interpreted using the ${\le}$ relation. We say that frames are the \emph{Lindenbaum--Tarski} algebras for geometric logic. We will explore this relationship further in \cref{sec:classifying_locales}.

Note that the lattice of open sets of any topological space forms a frame. Indeed, a topology on $X$ is nothing but a \emph{subframe} of the powerset of $X$ (i.e.\ a subset closed under finite meets and arbitrary joins).
Pointfree topology is the study of topology using frames directly, eschewing the underlying sets of points.

However, while we do not start with a predefined set of points, it is still possible to define points of a frame. From a logical perspective these correspond to \emph{models} of the geometric theory represented by the frame. Here a model is a consistent assignment of truth values to each verifiable property (saying which opens `contain' the point). Explicitly, we have the following conditions.
\begin{itemize}
 \item The top element $1$ must hold in a model.
 \item If $a$ holds and $a \le b$ then $b$ must hold too.
 \item If $a$ and $b$ both hold in the model, then so does $a \wedge b$.
 \item If $\bigvee_{\alpha \in I} a_\alpha$ holds, then $a_\alpha$ holds for some $\alpha \in I$.
\end{itemize}
In order-theoretic terms this defines what is known as a \emph{completely prime filter} on the frame, but we will simply call these \emph{points}.

If $\O X$ is the frame of opens of a topological space $X$, then every point of $X$ gives a point of $\O X$, but they might not be in one-to-one correspondence. For example, if $X$ contains two points that lie in precisely the same open sets, they will both correspond to the same point of $\O X$. However, in most cases of interest the points of $\O X$ can indeed be used to recover $X$ completely. If this is the case, we say that the space $X$ is \emph{sober}. Sobriety is a rather weak condition; classically, every Hausdorff space and every finite $T_0$ space is sober.

Frames can be understood as algebraic structures with operations $1$, $\wedge$ and a proper class of join operations of various arities.
A \emph{frame homomorphism} is simply a function that preserves these operations.

A continuous map of topological spaces induces a frame homomorphism between their frames of open sets \emph{in the opposite direction} by taking preimages.
This motivates the following definition.
\begin{definition}
 The category $\Loc$ of \emph{locales} is the opposite of the category $\Frm$ of frames and frame homomorphisms.
 A locale is simply a frame, but viewed from a topological, instead of an algebraic, perspective.
 
 To avoid confusion we will make a clear notational distinction between frames and locales: if $X$ is a locale we write $\O X$ for its correspond `frame of opens' and if $f\colon X \to Y$ is a morphism of locales we write $f^*\colon \O Y \to \O X$ for the corresponding frame homomorphism.
\end{definition}

Taking open sets gives a functor from $\Top$ to $\Loc$. This has a right adjoint, which sends a locale $X$ to its set of points equipped with the topology consisting of the open sets of the from $O_a = \{P \mid a \in P\}$,
and which sends a locale morphism $f\colon X \to Y$ to a map that acts on completely prime filters by taking preimages under $f^*$.
This is an idempotent adjunction which restricts to an equivalence between sober topological spaces and \emph{spatial} locales.
A locale is spatial if its opens are completely determined by the points they contain.

Not all locales are spatial. In classical mathematics non-spatial locales might appear to be rather unusual, but constructively locales that cannot be proved to be spatial are ubiquitous.

\subsection{Constructive mathematics}

The geometric logic introduced above is useful for reasoning about properties of points within a topological space or locale. In contrast, the classical logic that you are familiar with is more far-reaching, and can be used to reason about all of classical mathematics (which itself involves dealing with many topological spaces). \emph{Intuitionistic logic} is what replaces classical logic in this role when doing constructive mathematics.

As we did with our discussion of topology above, before we take a look at intuitionistic logic, let us take a step back and consider what logic is about.
This will be helpful since it can be difficult for those who are only used to classical mathematics to understand what it could mean to weaken some of the assumptions they feel obviously hold.
The assumption that constructive mathematics weakens is the \emph{law of the excluded middle}: $\phi \lor \neg \phi$, which is deeply ingrained in the way most mathematicians think to the extent that it can be very difficult for them to tell when they are using it.

One way to understand classical logic is as the logic describing Platonic truth. If mathematical objects truly exist, then how could they not have definite properties? Surely a proposition either holds of such an object or does not (though whether we can prove which is the case is another matter entirely). This is fine as far as it goes, but it is not the only thing we might want to use logic to do.

Non-classical logics can be understood as describing something different. We have already seen an example of this: geometric logic describes \emph{verifiable truths}.
There are also non-classical logics for knowledge, possibilities, fuzzy concepts and resources (see \cite{bell2001logical} or Chapter 2 of \cite{gabbay2016elementary} for more details).

Constructive mathematics uses \emph{intuitionistic logic}, which is often understood as describing things \emph{we already know to be true}. From this perspective $\phi \lor \neg \phi$ would mean that either we know that $\phi$ holds, or we know that assuming $\phi$ leads to a contradiction (since $\neg \phi$ means $\phi \heyting \bot$). Now it is not hard to see why it would not be reasonable to accept this as always being true.
Do note that constructive mathematicians do \emph{not} assert that excluded middle is false either, since that would mean that if we ever did find out that $\phi$ was true or $\phi$ was false we would have a contradiction. For further details on constructive mathematics see \citet*{bauer2013intuitionistic,bauer2017stages} and \citet*{vanDalen2001intuitionistic}.

We remark that the law of non-contradiction $\neg (\phi \wedge \neg \phi)$ is still valid constructively, since it just says that from $\phi$ and $\phi \heyting \bot$ we can conclude $\bot$. From a classical perspective one might feel that this is equivalent to excluded middle, but to derive this equivalence would require using either excluded middle itself or some other equivalent axiom.

One logical axiom that is equivalent to excluded middle is double negation elimination: $\neg\neg \phi \implies \phi$. This means that it is difficult to get rid of negations in constructive mathematics once they have been introduced and so it is important to define basic concepts without using negations if possible --- they should be defined by \emph{positive formulae}.

For example, asking that a set $S$ is nonempty ($S \ne \emptyset$) is a rather weak condition constructively and is often not terribly useful. We usually instead ask that the set is \emph{inhabited} --- that is, $\exists x \in S$. Now we can use the element $x$ to more easily prove further results. (On the other hand, if a set is not inhabited then it \emph{is} empty.)

It is often said that proof by contradiction is not permitted in constructive mathematics, but this claim should be interpreted carefully. We can, of course, prove $\neg\phi$ by assuming $\phi$ and deriving a contradiction, since this is the definition of $\neg\phi$. On the other hand, assuming $\neg \phi$ and deriving a contradiction only proves $\neg\neg \phi$ and this is not generally enough to conclude $\phi$.

The \emph{axiom of choice} is another famously nonconstructive axiom due to its claim that choice functions exist without giving any indication of how to actually construct one.
It can be shown that the axiom of choice implies excluded middle (for example, see \cite[Theorem 1.3]{bauer2017stages}), though it is much stronger. On the other hand, countable choice and dependent choice can be assumed without forcing excluded middle to hold, but we do not consider them to be entirely constructive either.
A version of \emph{finite} choice can be proved constructively by induction.

\subsubsection{Some models of constructive logic}

Another good way to understand intuitionistic logic is through its \emph{models}. This can be formalised as interpreting the logic in different toposes, but we can simply understand them as different logical universes in which to do mathematics. In order to make sense of a statement in intuitionistic logic it is often helpful to consider what it might mean in one or more of these different models.

One rather intuitive model corresponds to computability by Turing machines. This is related to the Brouwer–-Heyting–-Kolmogorov interpretation, which asserts that a logical formula is true if it has a \emph{realiser}, which can be defined recursively based on its syntactic form. For example, a realiser for a disjunction is simply a choice of one of the disjuncts together with a realiser for that. A realiser for an implication $\phi \implies \psi$ is given by a procedure that converts a realiser for $\phi$ into a realiser for $\psi$. (See \citet*{bauer2013intuitionistic} for more details.) Here we take these realisers be given by Turing machines.
If $\phi(n)$ means that the $n^\text{th}$ Turing machine halts then the formula $\forall n \in \N.\ \phi(n) \vee \neg\phi(n)$ would be interpreted as saying that the halting problem is decidable, and thus it is false in this model.\footnote{ %
Note that for any \emph{given} proposition $\psi$ we still have $\neg\neg(\psi \vee \neg \psi)$ as we argued above. It is the universally quantified formula involving many different propositions at once that is false.}

Since intuitionistic logic is more general than classical logic our usual understanding of classical logic also gives a model of intuitionistic logic. It can be helpful to bear in mind that constructive theorems cannot contradict classical results. On the other hand, this also exemplifies the fact that one model might not necessarily be enough for a complete understanding.

Our final class of models arises from topology and corresponds to the concept of \emph{local truth}, where the truth value of propositions is allowed to vary from place to place over some space.
In this model truth values correspond to opens in some fixed locale $X$. The proposition $\top$ means true \emph{everywhere}, the proposition $\bot$ means true \emph{nowhere} and conjunctions and disjunctions correspond to meets and joins.

Implication is more subtle. Since operation $a \wedge (-)$ on the frame $\O X$ preserves joins, it has a right adjoint $a \heyting (-)$ making $\O X$ into a Heyting algebra. We use this Heyting implication to interpret logical implication. If $X$ is spatial then $U \heyting V$ is given by the \emph{interior} of the set $\{x \in X \mid x \in U \implies x \in V\}$, which is the closest approximation to what we might naively expect while still giving an open set.
Note that $U \lor \neg U$ is only equal to $\top$ if $U$ has a \emph{complement} --- that is, if $U$ is clopen. Thus, excluded middle seldom holds in such models.

\begin{remark}
 There are some definite similarities between intuitionistic logic as it is used to describe local truth and geometric logic. However, they are not the same.
 Most evidently intuitionistic logic involves an implication connective, but geometric logic has none. This is related to the fact that the Heyting implication in a frame is not preserved by frame homomorphisms.
 We will see that topology does have its own way to deal with negation, however. There the complement of an open set is not another open set, but something new: a \emph{closed} set (corresponding to a \emph{refutable} property).
 Higher-order intuitionistic logic also permits existential and universal quantification over any sets, while the status of quantifiers in geometric logic is more subtle.
 In summary, intuitionistic logic has greater expressive power at the expense of having a weaker link to topological notions.
\end{remark}

\subsubsection{Truth values}

Classically there are precisely two truth values: $\top$ (true) and $\bot$ (false).
Constructively, the situation is more subtle. We write $\Omega$ for the set of truth values. Certainly, we still have $\bot,\top \in \Omega$. If $\phi$ is a proposition with no free variables, we write $\llbracket \phi \rrbracket$ for the truth value of $\phi$. Then $\Omega$ is ordered by $\llbracket \phi \rrbracket \le \llbracket \psi \rrbracket$ if and only if $\phi \implies \psi$ and the logical connectives give it the structure of a Heyting algebra. In particular, binary meet is given by logical conjunction and binary join is given by logical disjunction.

Note that since $\phi \iff (\phi \iff \top)$, we have $p = \llbracket p = \top \rrbracket$ for any $p \in \Omega$.
So truth values are completely determined by `whether they equal $\top$'.
This can be used to show that there is still a sense in which $\Omega$ does not have more than two elements.

\begin{lemma}
 For $p \in \Omega$, if $p \ne \top$ then $p = \bot$.
\end{lemma}
\begin{proof}
 Suppose $\llbracket\phi\rrbracket \ne \top$. This means $\neg(\llbracket\phi\rrbracket = \top)$. Hence, $\neg \phi$ holds and so $\llbracket \phi\rrbracket \le \bot$. Thus, $\llbracket \phi\rrbracket = \bot$.
\end{proof}

On the other hand, if $p \ne \bot$ we can only show $\neg\neg (p = \top)$. We cannot prove that $\Omega$ only contains $\top$ and $\bot$.
\begin{lemma}
 The equality $\Omega = \{\bot, \top\}$ is equivalent to excluded middle. 
\end{lemma}
\begin{proof}
 Simply note that $\llbracket \phi\rrbracket \in \{\bot, \top\} \iff \neg\phi \vee \phi$.
\end{proof}
\begin{remark}
 A proof that a statement implies some nonconstructive principle is called a \emph{Brouwerian counterexample}. While assuming such a statement holds might not lead to a contradiction, it cannot have a constructive proof and will fail in any model of mathematics where the nonconstructive principle in question does not hold.
\end{remark}

A proposition $\phi$ for which $\phi \vee \neg\phi$ is true is said to be \emph{decidable}. Constructively there is still some use for the set $2 = \{\bot, \top\}$ of decidable truth values.

The defining property of $\Omega$ is that it classifies subsets. If $X$ is a set and $\chi\colon X \to \Omega$, then $\{x \in X \mid \chi(x) = \top\}$ defines a subset of $X$.
Conversely, if $S \subseteq X$ then we can define the characteristic map $\chi_S\colon X \to \Omega$ by $x \mapsto \llbracket x \in S\rrbracket$. These correspondences are inverses and so $\Omega^X$ may be identified with the \emph{powerset} of $X$.

\begin{remark}
 In particular, note that functions into $\Omega$ are completely determined by which elements they send to $\top$.
 Furthermore, since the above correspondences are monotone, we have $f \le g$ for two such maps if and only if $f(x) = \top \implies g(x) = \top$.
\end{remark}

On the other hand, the exponential $2^X$ contains only the \emph{decidable subsets of $X$}.
\begin{definition}
 A subset $S \subseteq X$ is said to be \emph{decidable} if $\forall x \in X.\ x \in S \vee x \notin S$.
\end{definition}
Decidable subsets are analogous to clopen subsets in topology or to complemented elements in a frame. Indeed, they are precisely the complemented elements of $\Omega^X$, with the complement being given by $\{x \in X \mid x \notin S\}$.

\begin{definition}
 When the equality relation on a set $X$ is a decidable subset of $X \times X$ we say that $X$ has \emph{decidable equality}.
\end{definition}
The natural numbers $\N$ and the rationals $\Q$ both have decidable equality, but this is not true for $\Omega$ unless excluded middle holds.

The lattice $\Omega$ also plays an important role in pointfree topology.
\begin{lemma}
 The lattice $\Omega$ is frame, with the arbitrary joins given by $\bigvee S = \llbracket \top \in S\rrbracket$. Moreover, it is the \emph{initial} frame: the unique frame map ${!}^*\colon \Omega \to \O X$ sends $p \in \Omega$ to $\bigvee\{1 \mid p = \top\}$.
\end{lemma}
\begin{proof}
 We already know $\Omega$ is a Heyting algebra. So $a \wedge (-)$ has a right adjoint and hence preserves existing suprema. Thus, it to show $\Omega$ is a frame, it suffices to show it admits arbitrary joins.
 We now prove that the join of $S \subseteq \Omega$ is given by $\llbracket \top \in S\rrbracket$.
 
 To see that this is an upper bound for $S$ we take $s \in S$ and show $s \le \llbracket \top \in S\rrbracket$. If $s = \top$ then $\top = s \in S$, so that $s = \top \implies \top \in S$ and the desired inequality holds.
 
 To see that $\llbracket \top \in S\rrbracket$ is the \emph{least} upper bound we must show that for any upper bound $u \in \Omega$, we have $\top \in S \implies u = \top$.
 Suppose $u$ is an upper bound of $S$ and $\top \in S$. Then in particular, $\top \le u$ and so $u = \top$, as required.
 
 Now we show that $\Omega$ is the initial frame.
 Let $\O X$ be a frame. Any frame map from $\Omega$ to $\O X$ must send $\top$ to $1$ and preserve joins. But now for any $p \in \Omega$, we have that $p = \llbracket p = \top\rrbracket = \bigvee\{ \top \mid p = \top\}$\footnote{The set $\{\top \mid p = \top\}$ might seem a little strange from the classical perspective. It is the subset of $\{\top\}$ corresponding to the constant map sending $\top$ to the element $\llbracket p = \top \rrbracket$ in $\Omega$. Classically, we seldom consider subsets defined by propositions with no free variables since any such subset is either the entire set or empty.} and so the only possible map is ${!}^*\colon p \mapsto \bigvee \{ 1 \mid p = \top \}$. Finally, it is not hard to check that this is indeed a frame homomorphism.
\end{proof}

\begin{remark}
 Interestingly, the set $2 = \{0,1\}$ cannot be proved to be a frame constructively. Joins of decidable subsets of $2$ always exist, but joins of arbitrary subsets might not.
 Indeed, if $2$ is complete then $\bigvee\{1\mid \phi\} = 0 \iff {\bigvee\{1\mid \phi\} \le 0} \iff (\phi \implies 1 \le 0) \iff \neg \phi$. But $2$ has decidable equality and so either $\bigvee\{1\mid \phi\} = 0$ or $\bigvee\{1\mid \phi\} \ne 0$. Hence, we conclude that $\neg\phi \vee \neg\neg\phi$. This logical principle is weaker than excluded middle, but it is nevertheless nonconstructive.
\end{remark}

Recall that \emph{points} of a frame $\O X$ are given by consistent assignments of its elements to truth values. We quickly see that these are simply frame homomorphisms from $\O X$ to $\Omega$.
Since $\Omega$ is the initial frame, it corresponds to the \emph{terminal locale} $1$ and so points of a locale $X$ are simply morphisms from $1$ to $X$. Thus, our definition coincides with the usual definition of point from category theory.

Note that $1$ has a unique point by virtue of being terminal and $\Omega$ is isomorphic to the powerset of $1 = \{*\}$ and hence to the frame of opens of the unique one-point topological space.
So we can identify the spatial locale $1$ with the sober space $1$. More generally, we can show that locales faithfully represent sets (viewed as discrete spaces).
The frames of opens of such \emph{discrete locales} are given by powerset lattices.

\subsubsection{Why work constructively?}\label{sec:why_work_constructively}

Constructive proofs often require more care than classical ones would and so it is natural to ask why it is worth the trouble.
Probably the most convincing argument is the much wider applicability of constructive proofs: they hold in any topos (with natural numbers object). Interpreting the results in these models in terms of external concepts then gives rise to new and different results depending on which topos we use.

We stress that it is \emph{not} necessary to have a deep understanding of topos theory to do constructive mathematics, just as one can do classical mathematics without deep knowledge of set theory.
When it comes to interpreting a result in a particular topos, some basic knowledge of that topos and of categorical logic in general is helpful,
though even then, if one knows how to interpret basic structures, it is usually possible to fill in the details. This is the approach we will take in these notes. For further information see \citet*[Part D]{elephant2}.

Le us consider an example of how interpreting results in a topos can be useful. Consider the topos $\Sh(B)$ of sheaves over a locale $B$. The internal logic of this topos corresponds to `local truth' over $B$.
It can be shown that a locale internal to $\Sh(B)$ corresponds externally to a locale morphism into $B$ (see \citet*[Section C1.6]{elephant2}).
Thus, if we interpret a constructive theorem about locales in this topos, we obtain an analogous result about locale morphisms into $B$. Classically, there is an entire field of \emph{fibrewise topology} that studies such a situation (for topological spaces), but by working constructively we can deduce many of these results automatically. We will see a more specific example of this in \cref{sec:fibrewise_application}.

Of course, interpreting the result in the category $\Set$ (in a classical metatheory) simply recovers the classical result.
Some other important examples are shown in \cref{tab:interpretations}. If we are primarily interested in interpreting the results in a particular class of toposes, we might be able to make use of certain classical principles, but for the most general results we should not assume any of these.

\begin{table}[H]
\begin{tabular}{>{\raggedright\arraybackslash}m{2.25cm}>{\raggedright\arraybackslash}m{5.75cm}>{\raggedright\arraybackslash}m{5.5cm}} %
  \toprule
  \textbf{Topos} & \textbf{Interpretation} & \textbf{Permitted classical axioms} \\
  \midrule
  $\Set$ & Classical results & Axiom of choice \\
  \rowcolor{gray!20}
  $G\textsf{-}\Set$ & $G$-equivariant topology & Excluded middle \\
  $\mathrm{Eff}$ & Computable analysis & Dependent choice \\
  \rowcolor{gray!20}
  $\Sh(B)$ & Fibrewise topology over $B$ & --- \\
  \bottomrule
\end{tabular}
\caption{Interpreting constructive topology in different toposes gives results in a number of subfields for free.
Here $B$ is a locale (or sober space) and $G$ is any \'etale-complete localic group (such as a discrete or profinite group, see \cite{moerdijk1988classifying}).} \label{tab:interpretations}
\end{table}

\section{Frame presentations and algebraic techniques}

A convenient aspect of pointfree topology is that frames are simply a kind of algebraic structure.
There is a subtlety in that they have a proper class of operations, but we can nonetheless show that the forgetful functor from $\Frm$ to $\Set$ has a left adjoint and $\Frm$ has coequalisers.\footnote{We remark that the existence of a left adjoint is enough to conclude any equationally presented category is monadic over $\Set$.
However, constructively this alone is not sufficient to deduce cocompleteness. Nonetheless, cocompleteness does follow as long as we have sufficiently many coequalisers. For instance, see \citet*[Proposition 5.6.11]{riehl2017category}.}
Thus, most of the basic ideas from universal algebra are still applicable.
In particular, $\Frm$ is complete and cocomplete and we can talk about subframes and quotient frames in the usual way.

A particularly important consequence is that we can present frames by generators and relations.
A \emph{presentation} has the form $\langle G \mid R\rangle$ where $G$ is a set of generators and $R$ is a set of equalities\footnote{We will also find it convenient to use \emph{inequalities} as relations, but since $a \le b \iff a \wedge b = a$ (or alternatively, $a \le b \iff a \vee b = b$), these can be replaced by equivalent equalities.} between formal combinations of the generators, called relations.
Formally, $R$ can be defined to be a subset of $F G \times F G$ where $F G$ is the free frame on $G$. The frame defined by a presentation can be constructed as the quotient of $F G$ by the congruence generated by $R$.
\par %
The function from $G$ to $\langle G \mid R \rangle$ sending generators to their image in the quotient of $F G$ satisfies the following universal property: any function $f$ from $G$ to a frame $L$ such that the images of the generators under $f$ satisfy the relations in $R$ induces a unique frame homomorphism $\overline{f}\colon \langle G \mid R\rangle \to L$ such that the following diagram commutes.
\begin{center}
 \begin{tikzpicture}[node distance=2.45cm, auto]
  \node (FG) {$\langle G \mid R \rangle$};
  \node (G) [below of=FG] {$G$};
  \node (Y) [right of=G] {$L$};
  \draw[->] (G) to node {} (FG);
  \draw[->, dashed] (FG) to node {$\overline{f}$} (Y);
  \draw[->] (G) to node [swap] {$f$} (Y);
 \end{tikzpicture}
\end{center}

We will see that this provides a very convenient way to construct spaces which is completely unavailable in point-set topology.

\subsection{Classifying locales}\label{sec:classifying_locales}

Presentations can be understood as describing (propositional) geometric theories.
The generators correspond to basic propositions and the relations correspond to axioms.
A sequent $\phi \vdash \psi$ corresponds to the relation $\phi \le \psi$ (or as an equality $\phi \wedge \psi = \psi$).
We call the locale corresponding to this presentation the \emph{classifying locale} of the geometric theory.

As mentioned before, points of the locale correspond to models of the geometric theory. The universal property of the presentation corresponds to the fact that we can show something is a model just by checking the axioms.
A point of $\langle G \mid R\rangle$ is simply a frame homomorphism from $\langle G \mid R\rangle$ to $\Omega$, which corresponds to an assignment of the generators to truth values that respects the relations.

As a very simple example, points of the free frame on one generator correspond to the elements of $\Omega$.
It turns out that this frame is spatial, and the topology on $\Omega$ is generated by the single open $\{\top\}$. This space (as well as the corresponding locale) is called \emph{Sierpiński space}.

We usually conceptualise geometric theories as describing their hypothetical models --- that is, the points of the locale. This is despite the fact that we do not have a completeness theorem for geometric logic in the usual sense and so nontrivial theories might even fail to have any models at all.\footnote{The completeness theorem is restored if we look at models not only in $\Set$, but in all Grothendieck toposes. Indeed, one way to understand locales is as a succinct way to encode what the points are from the perspective of every Grothendieck topos at once.}
Moreover, the theory alone describes not only the points of the locale, but also the `topology'.

It will be helpful to look at an example. We will describe a theory that gives the locale of real numbers. Recall that real numbers can be constructed using Dedekind cuts and so let us consider the theory of Dedekind cuts on the rationals.
Such a \emph{Dedekind cut} is a pair $(L,U)$ of subsets of $\Q$ satisfying a number of conditions. The intuition is that we expect $L = \{q \in \Q \mid q < \rho\}$ and $U = \{q \in \Q \mid q > \rho\}$ for some real number $\rho$. The precise conditions are as follows.
\begin{enumerate}[1)]
 \item $L$ is downward closed: if $q \in L$ and $p \le q$ then $p \in L$
 \item $L$ is \emph{rounded}: if $p \in L$ then $q \in L$ for some $q > p$
 \item $L$ is inhabited: there is some $q \in L$ (to rule out $\rho = -\infty$)
 \item $U$ is upward closed: if $p \in U$ and $p \le q$ then $q \in U$
 \item $U$ is rounded: if $q \in U$ then $p \in U$ for some $p < q$
 \item $U$ is inhabited: there is some $q \in U$ (to rule out $\rho = \infty$)
 \item $L$ and $U$ are disjoint: if $p \in L$ and $q \in U$ then $p < q$
 \item The pair $(L,U)$ is \emph{located}: if $p < q$ then either $p \in L$ or $q \in U$
\end{enumerate}
These can be formulated as a geometric theory. We define for each $q \in \Q$ a basic proposition $\ell_q$, which we interpret to mean ``$q \in L$'', and a basic proposition $u_q$, which we interpret to mean ``$q \in U$''.
In terms of these the conditions become the following axioms.
\begin{displaymath}
\begin{array}{l@{\qquad\quad}r@{\hspace{1.5ex}}c@{\hspace{1.5ex}}l@{\quad}@{}l@{\qquad\quad}r@{}}
 (1) & \ell_q &\vdash& \ell_p & \text{ for $p \le q$} & \text{($L$ downward closed)} \\
 (2) & \ell_p &\vdash& \bigvee_{q > p} \ell_q & \text{ for $p \in \Q$} & \text{($L$ rounded)} \\
 (3) & \top &\vdash& \bigvee_{q \in \Q} \ell_q && \text{($L$ inhabited)} \\
 (4) & u_p &\vdash& u_q & \text{ for $p \le q$} & \text{($U$ upward closed)} \\
 (5) & u_q &\vdash& \bigvee_{p < q} u_p & \text{ for $q \in \Q$} & \text{($U$ rounded)} \\
 (6) & \top &\vdash& \bigvee_{q \in \Q} u_q && \text{($U$ inhabited)} \\
 (7) & \ell_p \land u_q &\vdash& p < q & \text{ for $p,q \in \Q$} & \text{($L$ and $U$ disjoint)} \\
 (8) & \top &\vdash& \ell_p \lor u_q & \text{ for $p < q$} & \text{(locatedness)}
\end{array}
\end{displaymath}
Note that the models of this theory are indeed the Dedekind cuts of $\Q$ (if for each model we take $L$ to be the set of $q \in \Q$ such that $\ell_q$ holds and $U$ to be the set of $q \in \Q$ such that $u_q$ holds).

This (description of the) geometric theory can then be translated to give a frame presentation as above. The basic propositions $\ell_q$ and $u_q$ become generators and the axioms become relations. For instance, $\top \vdash \bigvee_{q \in \Q} \ell_q$ gives the relation $1 \le \bigvee_{q \in \Q} \ell_q$. Axiom schema (7) deserves particular mention; it involves the metatheoretical truth value $p < q$, which can be understood in geometric logic by interpreting it as a shorthand for the disjunction $\bigvee\{\top \mid p < q\}$. In this way we see the corresponding relations are $\ell_p \land u_q \le {\bigvee\{1 \mid p < q\}}$. Note that in the resulting frame $p < q$ becomes the open ${!}^*(\llbracket p < q \rrbracket)$ where ${!}^*$ is the unique frame map from $\Omega$.

These relations can be simplified a little. For example, (1) and (2) together are equivalent to the equality $\ell_p = \bigvee_{q > p} \ell_q$. Also, note that since inequality of rational numbers is decidable, we can equivalently express the relations for (7) as $\ell_p \land u_q = 0$ for $p \ge q$ (while the condition is trivial in the case $p < q$). We arrive at the following proposition.
\begin{proposition}
 The classifying locale for the theory of Dedekind cuts has underlying frame
 \begin{align*}
  \O \R = \langle \ell_q, u_q,\ q \in \Q \mid {} & \ell_p = \bigvee_{q > p} \ell_q,\ u_q = \bigvee_{p < q} u_p, \\
                                                 & \bigvee_{q \in \Q} \ell_q = 1,\ \bigvee_{q \in \Q} u_q = 1, \\
                                                 & \ell_p \wedge u_q = 0\, \textnormal{ for $p \ge q$},\\
                                                 & \ell_p \vee u_q = 1\, \textnormal{ for $p < q$} \rangle.
 \end{align*}
\end{proposition}
This is called the \emph{locale of reals}. The elements $\ell_q$ and $u_p$ can be interpreted as the subbasic open sets $(q, \infty)$ and $(-\infty, p)$. The points of the locale are real numbers by construction: if $p\colon 1 \to \R$ is a locale map then $L = \{q \in \Q \mid p^*(\ell_q) = \top\}$ and $U = \{q \in \Q \mid p^*(u_q) = \top\}$ define a Dedekind cut.
The induced topology on this set of points is the usual one.
However, the locale of reals is not necessarily spatial constructively!\footnote{An example of a model of constructive mathematics where $\R$ fails to be spatial is the effective topos $\mathrm{Eff}$ from computability theory. See \cref{rem:spatial_interval_noncompact} and \citet*{bauer2006konig} for more information.} The failure of such an important locale to be spatial is a big reason why the pointfree approach is so important in constructive mathematics.

Note that by describing the Dedekind cuts as a geometric theory and finding the classifying locale, we didn't just get the correct set of points; we also got the right topology `for free'! Moreover, we obtained the full locale even though it is not necessarily determined by its (global) points.

We end this subsection with a few additional examples.
\begin{example}[Cantor space]\label{ex:cantor_space}
Cantor space is the locale of binary sequences. We can verify if the bit at any given index is $0$ or $1$, giving generators $z_n$ and $u_n$ respectively.
Exactly one of these two options holds, and so we have the presentation \[\O(2^\N) \cong \langle z_n, u_n,\ n \in \N \mid z_n \wedge u_n = 0, z_n \vee u_n = 1\rangle.\]
As expected, the points of this locale correspond to decidable subsets of $\N$.
\end{example}

\begin{example}[The Stone spectrum]\label{ex:stone_spectrum}
Let $L$ be a bounded distributive lattice. The Stone spectrum of $L$ is the space of \emph{prime filters} of $L$.
A prime filter is a subset $F \subseteq L$ such that
\begin{itemize}
 \item if $a \in F$ and $a \le b$ then $b \in F$,
 \item $1 \in F$,
 \item if $a \in F$ and $b \in F$ then $a \wedge b \in F$,
 \item $0 \notin F$,
 \item if $a \vee b \in F$ then $a \in F$ or $b \in F$.
\end{itemize}
Writing $a^f$ for a basic proposition meaning $a \in F$ this can be expressed as a geometric theory.
The resulting frame presentation is then
\[\langle a^f,\ a \in L \mid 1^f = 1, (a \wedge b)^f = a^f \wedge b^f, 0^f = 0, (a \vee b)^f = a^f \vee b^f\rangle.\]
The Zariski spectrum of a commutative ring can be defined similarly.
In classical mathematics the spatiality of all Stone or Zariski spectra is equivalent to the Boolean prime ideal theorem (a consequence of the axiom of choice), but we can always define the appropriate locale. With the pointfree approach Stone duality is constructively valid.
\end{example}
\begin{remark}
 Analogously to the above example, if $L$ is a frame then we can recover $L$ as the (frame of opens of) the classifying locale of the theory of completely prime ideals of $L$.
\end{remark}

\begin{example}[Surjections from $\N$ to $X$]
Fix an inhabited set $X$ and consider the following geometric theory. There is a basic proposition denoted by $[f(n) = x]$ for each $n \in \N$ and $x \in X$. The axioms are:
\begin{itemize}
 \item $[f(n) = x] \wedge [f(n) = y] \vdash x = y$ for $n \in \N$ and $x,y \in X$,
 \item $\top \vdash \bigvee_{x \in X} [f(n) = x]$ for $n \in \N$,
 \item $\top \vdash \bigvee_{n \in \N} [f(n) = x]$ for $x \in X$.
 \end{itemize}
The models of such a theory are given by surjective functions $f\colon \N \twoheadrightarrow X$ where the proposition $[f(n) = x]$ is interpreted to mean that $f(n) = x$.
The first axiom schema ensures the relation is functional, the second that it is total, and the third imposes surjectivity.

This locale can be proved to be nontrivial (see \citet*[Example C1.2.8]{elephant2}).
However, if we take $X = \Omega^\N$ then there are no surjections from $\N$ to $X$ by Cantor's theorem and thus the locale has no points whatsoever!
The nontriviality of this locale is related to the fact that such surjections can appear after passing to a forcing extension.
This example is arguably rather pathological, but a similar locale has found use in topos theory (see \citet*{galoisTheoryGrothendieck}).
It is remarkable that an intuitive axiomatisation of a theory that has no models has lead to a complete characterisation of a nontrivial locale.
\end{example}

\subsection{Free frames}\label{sec:free_frames}

To show that frames are algebraic, and in particular to be able to present them by generators and relations, we must prove the existence of free frames.
This means constructing a left adjoint to the forgetful functor $U\colon \Frm \to \Set$.

It will be helpful to proceed in two steps. The forgetful functor $U$ factors through the category $\SLat{\wedge}$ of $\wedge$-semilattices.
\begin{definition}
 Recall that a \emph{semilattice} is simply a commutative idempotent monoid. Such a structure induces a natural order structure by $a \le b \iff ab = b$
 and with respect to this order $ab$ gives the greatest lower bound of $\{a,b\}$. To emphasise this viewpoint we write $ab$ as $a \wedge b$ and call the structure a $\wedge$-semilattice.
 We could have equally defined the order in the reverse direction, in which case we write $ab$ as $a \vee b$ and call it a $\vee$-semilattice.
\end{definition}

Any frame can be viewed as a $\wedge$-semilattice by simply forgetting the join operations. Since taking adjoints respects composition, the adjoint to $U\colon \Frm \to \Set$ can be constructed as the composition of the adjoint $\D\colon \SLat{\wedge} \to \Frm$ of this forgetful functor and the free semilattice functor, which can be constructed just as for any other finitary algebra.
Thus, it just remains to construct the free frame on a $\wedge$-semilattice.

\subsubsection{Frames of downsets}

We need to freely add joins to a $\wedge$-semilattice $M$. A formal join of elements of $M$ will be specified by some subset $S \subseteq M$. By idempotence, a join of such formal joins is again one of these formal joins, and we do not need to iterate this procedure. But distinct subsets of $M$ do not always give distinct formal joins. For instance, if $a \le b$ then $\bigvee\{a,b\}$ should be the same as $\bigvee\{b\} = b$.

\begin{definition}
 A subset $S$ of a poset $P$ is called a \emph{downset} if it is downward closed. If $a \in P$ we write ${\downarrow} a$ for the principal downset generated by $a$, which is given by $\{b \in P \mid b \le a\}$.
\end{definition}
We expect that the join of a subset $S$ should be completely specified by the downset generated by $S$. It turns out there are no additional conditions and we have the following result.

\begin{proposition}
 The free frame on a $\wedge$-semilattice $M$ is given by the frame $\D M$ of downsets of $M$ (ordered by subset inclusion).
\end{proposition}
\begin{proof}
 We first note that $\D M$ is indeed a frame, since it is a subframe of the powerset of $M$.
 To show it is a free frame on $M$ we must show there is a $\wedge$-semilattice homomorphism $\eta_M \colon M \to \D M$ such that for any $\wedge$-semilattice homomorphism $f$ from $M$ to a frame $L$ there is a unique frame homomorphism $f^\flat\colon \D M \to L$ for which the following diagram commutes.
 \begin{center}
   \begin{tikzpicture}[node distance=2.5cm, auto]
    \node (FM) {$\D M$};
    \node (M) [below of=FM] {$M$};
    \node (L) [right of=M] {$L$};
    \draw[->] (M) to node {$\eta_M$} (FM);
    \draw[->, dashed] (FM) to node {$f^\flat$} (L);
    \draw[->] (M) to node [swap] {$f$} (L);
   \end{tikzpicture}
 \end{center}
 The map $\eta_M$ sends $m$ to ${\downarrow} m$, which is quickly seen to preserve finite meets.
 Any map $f^\flat$ as above must send ${\downarrow} m$ to $f(m)$ and so by join-preservation we have \[f(D) = f^\flat(\bigcup_{d \in D} {\downarrow} d) = \bigvee_{d \in D} f(d).\]
 Thus, this condition specifies $f^\flat$ uniquely. It remains to show $f^\flat$ so defined is indeed a frame homomorphism.
 
 It is easy to see it preserves joins and $1$. Now consider $D, E \in \D M$. We have
 \begin{align*}
  f^\flat(D) \wedge f^\flat(E) &= \ \bigvee_{\mathclap{d \in D, e \in E}} f(d) \wedge f(e) \\
                               &= \ \bigvee_{\mathclap{d \in D, e \in E}} f(d \wedge e) \\
                               &\le \ \bigvee_{\mathclap{a \in D \cap E}} f(a) \\
                               &= f^\flat(D \cap E),
 \end{align*}
 where the second last inequality holds since $d \wedge e \le d \in D$ and $d \wedge e \le e \in E$ give that $d \wedge e \in D \cap E$.
 The reverse inequality holds by monotonicity and so $f^\flat$ preserves binary meets.
\end{proof}

\subsubsection{Free semilattices and finiteness}

The above argument suffices to prove that the free frame on a set exists. However, to better understand the structure of free frames it will be useful to have a more concrete description of the free $\wedge$-semilattice on a set.
We will actually find it more convenient to work with $\vee$-semilattices, but of course, we obtain the corresponding $\wedge$-semilattices by simply reversing the order.

The free semilattice on a set $G$ should be the quotient of the free commutative monoid on $G$ by a congruence enforcing idempotence. Classically, it is easy to see that this is isomorphic to the finite powerset $\Pfin(G)$ under union. Constructively, finiteness can be quite subtle, bifurcating into a number of different inequivalent notions. Let us take this opportunity to discuss some of them.

We start with what is perhaps the most obvious notion of finiteness.
\begin{definition}
 A \emph{finite cardinal} is a set of the form $[n] \coloneqq \{i \in \N \mid i < n\}$ for some $n \in \N$.
 A set is said to be \emph{Bishop-finite} if it is isomorphic\footnote{Here by isomorphic, we mean that ``there exists an isomorphism'' in the sense of higher-order intuitionistic logic. In type theory one can express a more restrictive definition of isomorphism, but this is beyond the scope of these notes.} to a finite cardinal.
\end{definition}
Bishop-finite sets have decidable equality and are closed under binary (and even Bishop-finite) products and coproducts, as well as under taking decidable subsets.

However, they are \emph{not} necessarily closed under arbitrary subsets.
Indeed, a subsingleton $\{{*} \mid p = \top\}$ is only Bishop-finite if $p = \bot \lor p = \top$.
This can be quite upsetting for classical mathematicians, though arguably their intuition only applies to decidable subsets.
Perhaps it helps to consider a model: in the topos $\Sh(B)$ of sheaves over a locale $B$ the Bishop-finite sets correspond to finite covering spaces over $B$ (i.e.\ covering spaces with finite fibres). Then the failure of closure under subsets simply means that an open subspace of a finite covering space can fail to be a finite covering space.

Surprising aspects of constructive mathematics can often be understood by generalising from sets to locales. We will see that finiteness is closely related to compactness, and it is no surprise that (open) subspaces of compact spaces might fail to be compact.

Most importantly for us, Bishop-finite subsets are not closed under taking unions.
For example, we cannot decide if the union of $\{\top\}$ and $\{p\}$ in $\Omega$ has two elements or not unless $p = \top \lor p = \bot$.
For a good notion of finite powerset we will need a different definition.
\begin{definition}
 A set $X$ is said to be \emph{Kuratowski-finite} or \emph{finitely indexed} if it is a quotient of a finite cardinal --- that is, if there is a natural number $n$ and a surjection $e\colon [n] \twoheadrightarrow X$.
\end{definition}
Kuratowski-finite sets have better closure properties than Bishop-finite sets. In addition to being closed under decidable subsets, Bishop-finite products and Bishop-finite coproducts,
they are also closed under quotients (by construction) and under Kuratowski-finite unions. We give the proof of one these claims for illustrative purposes. The others are not difficult and make for good exercises.
\begin{lemma}
 A decidable subset of a Kuratowski-finite set is Kuratowski-finite.
\end{lemma}
\begin{proof}
 Let us first show that decidable subsets of $[n]$ are Bishop-finite by induction on $n$.
 This is trivial for $n = 0$. Now suppose $n = k+1$ and that $S$ is a decidable subset of $[n]$.
 Note that the set $S\setminus \{k\}$ is a decidable subset of $[k]$. Thus, by the inductive hypothesis we have an isomorphism $\phi\colon [m] \to S \setminus \{k\}$.
 Now by decidability, either $k \in S$ or $k \notin S$. If $k \notin S$ then $S = S\setminus\{k\}$ and we are done. Otherwise,
 we can construct a map $\phi'\colon [m+1] \to S$ by
 \[\phi'(x) = \begin{cases}
               \phi(x) & \text{if $x < m$} \\
               k       & \text{if $x = m$}
              \end{cases},
 \]
 which is well-defined since $\N$ has decidable equality and is easily seen to be an isomorphism.
 
 Now consider a decidable subset $S$ of an arbitrary Kuratowski-finite set $X$. There is a surjection $e\colon [n] \twoheadrightarrow X$ for some $n \in \N$. Observe that $e^{-1}(S)$ is decidable subset of $[n]$ and hence Bishop-finite by the previous argument. Finally, we have $S = e(e^{-1}(S))$ by surjectivity. Hence, $S$ is the image of a Bishop-finite set and is therefore Kuratowski-finite itself.
\end{proof}
\begin{remark}
 In the above proof we defined the map $\phi$ by cases. There is some subtlety in ensuring such definitions are well-defined in constructive mathematics. The branches must be exhaustive and agree on any overlap. In this case the branching condition is decidable and so there is no problem, but we could not case match on either an arbitrary proposition or its negation holding, since this is not necessarily exhaustive.
\end{remark}

However, Kuratowski-finite sets are still not closed under arbitrary subsets. Indeed, the following result implies that if every subsingleton (i.e. set of the from $\{{*} \mid p = \top\}$) is Kuratowski-finite then we can prove excluded middle.
\begin{lemma}
 A Kuratowski-finite set $X$ is either empty or inhabited.
\end{lemma}
\begin{proof}
 We have a surjection $e\colon [n] \twoheadrightarrow X$ for some $n \in \N$. Either $n = 0$ or $n \ge 1$. In the former case, $[n] = [0]$ is empty and hence so is $X$. In the latter case, $e(0) \in X$ and so $X$ is inhabited.
\end{proof}

\begin{remark}
Kuratowski-finite sets are `finitely listable'. The difference between them and Bishop-finite sets is that Bishop-finite sets may be listed \emph{without repeats}. The inability to resolve repeats indicates one disadvantage of Kuratowski-finite sets: they might not have decidable equality (as with the set $\{\top, p\}$ we considered before). Indeed, this idea can be used to show that a set is Bishop-finite if and only if it is Kuratowski-finite and has decidable equality.
\end{remark}

We are now in position to link this discussion back to semilattices. We start by noting that $\vee$-semilattices admit all Kuratowski-finite joins.
\begin{lemma}
 Let $L$ be a $\vee$-semilattice and let $S \subseteq L$ be a Kuratowski-finite subset. Then $\bigvee S$ exists and is equal to $\bigvee_{i = 0}^{n-1} e(i)$ for any surjection $e\colon [n] \twoheadrightarrow S$.
\end{lemma}
\begin{proof}
 For each $s \in S$ there is a $j \in [n]$ such that $s = e(j) \le \bigvee_{i = 0}^{n-1} e(i)$. Hence $\bigvee_{i = 0}^{n-1} e(i)$ is an upper bound of $S$.
 
 Now suppose $u \in L$ is an upper bound of $S$. Then for each $j \in [n]$, we have $e(j) \le u$ and hence $\bigvee_{i = 0}^{n-1} e(i) \le u$. Thus, this is the least upper bound of $S$.
\end{proof}
\begin{corollary}\label{cor:semilattice_maps_preserve_K_finite_joins}
 A $\vee$-semilattice homomorphism preserves Kuratowski-finite joins.
\end{corollary}

Now as might be expected, we can show that a free semilattice is given by the set of Kuratowski-finite subsets.

\begin{proposition}
 The free semilattice on a set $G$ is given by the set $\Pfin(G)$ of Kuratowski-finite subsets of $G$ under union.
\end{proposition}
\begin{proof}
 Certainly, $\Pfin(G)$ is a $\vee$-semilattice.
 Now let $L$ be a $\vee$-semilattice and consider a function $f\colon G \to L$.
 This should correspond to a unique $\vee$-semilattice homomorphism $f^\flat \colon \Pfin(G) \to L$ such that $f^\flat(\{g\}) = f(g)$ for all $g \in G$.
 
 By \cref{cor:semilattice_maps_preserve_K_finite_joins} such a map must satisfy $f^\flat(S) = \bigvee_{s \in S} f(s)$.
 It remains to show that this indeed defines a map which preserves $0$ and $\vee$, but this easy to check.
\end{proof}

In summary, we obtain the following description of free frames.
\begin{proposition}
 The free frame on a set $G$ is given by $\D(\Pfin(G)\op)$.
\end{proposition}
So an element of a free frame is a set of finite subsets of the generating set. Interpreting the inner sets as formal finite meets and the outer sets as formal joins, we see that general elements are given by formal joins of formal finite meets. Indeed, any more complicated expression using the frame operations can be brought into this form using the frame distributivity law.

\subsection{Sublocales and frame quotients}\label{sec:sublocales}

To show that frame presentations exist, we require not only free frames, but also coequalisers.

\begin{definition}
 A \emph{congruence} on a frame $L$ is an equivalence relation on $L$ that is also a subframe of $L \times L$.
\end{definition}

The kernel equivalence relation $\{(u,v) \in L \times L \mid h(u) = h(v)\}$ of any frame homomorphism $h\colon L \to M$ is a congruence relation.
On the other hand, the quotient of a frame by a congruence is again a frame, and the quotient map is a frame homomorphism. The frame structure on the quotient requires some explanation.
It is easy to define finitary operations on a quotient by picking representatives as necessary, but in general this does not work for infinitary operations since we would need to use the axiom of choice to choose representatives. However, every congruence class of a frame congruence has a largest element (by taking the join of all the elements) and hence in this case we can pick canonical representatives constructively.

Congruences are closed under intersections and hence form a complete lattice. We say the smallest congruence containing a relation $R$ is the congruence generated by $R$ and denote it by $\langle R\rangle$. These can be used to construct coequalisers in $\Frm$. In particular, the frame presentation $\langle G \mid R\rangle$ can be constructed as the quotient $\D(\Pfin(G)\op) / \langle R\rangle$.
From presentations we can then construct arbitrary colimits in the usual way for algebraic structures.

\subsubsection{Sublocales}

Let us now consider how to understand quotient frames from a more geometric perspective.
Frame surjections correspond to regular epimorphisms in $\Frm$ and hence regular monomorphisms in $\Loc$.
Thus, a quotient frame is a kind of sub-locale.
Indeed, subspaces are precisely the regular subobjects in $\Top$.
Furthermore, if $S$ is a subspace of a topological space $X$, then restricting opens to $S$ gives a surjective frame homomorphism from $\O X$ to $\O S$.
Thus, it seems appropriate to call these \emph{sublocales}.

From a logical point of view, quotient frames correspond to adding additional axioms to the geometric theory.
These axioms are simply a generating set for the congruence associated to the quotient.
They add additional restrictions and hence reduce the number of models (i.e.\ points), which helps explain why sublocales act like subspaces.

\begin{remark}
 Instead of using frame congruences it is also common to represent sublocales using certain closure operators called \emph{nuclei} (obtained by composing a quotient map with its right adjoint) or certain subsets of the frame called \emph{sublocale sets} (which are the fixed points of nuclei). For more details on these alternative approaches see \citet*{picado2012book}.
\end{remark}

Classical topology suggests some important classes of sublocales. Firstly, for each element $a \in \O X$ we expect there to be an \emph{open sublocale}.
\begin{definition}\label{def:open_sublocales}
 Let $X$ be a locale and consider $a \in \O X$. The \emph{open sublocale} corresponding to $a$ is given by the frame quotient from $\O X$ to ${\downarrow} a$ sending $u$ to $u \wedge a$.
\end{definition}
Note that this accords with our topological intuition: taking the meet with $a$ restricts the open to the open subspace and opens in the subspace correspond to the opens in $X$ which are contained in $a$.

The kernel equivalence relation of such an open frame quotient is given by $\Delta_a = \{(u,v) \mid u \wedge a = v \wedge a\}$.
Logically, in the open sublocale the proposition $a$ becomes true. Thus, we should expect $\Delta_a$ to be generated by the pair $(a,1)$. This is indeed the case.
\begin{lemma}
 The open congruence $\Delta_a = \langle(a,1)\rangle$.
\end{lemma}
\begin{proof}
 Certainly, $(a,1) \in \Delta_a$ since $a \wedge a = a = 1 \wedge a$. Thus, $\langle (a,1) \rangle \subseteq \Delta_a$.
 
 Now suppose $u \wedge a = v \wedge a$. In the congruence $\langle (a,1)\rangle$, we have $u \sim u$ and $a \sim 1$ so that $u \wedge a \sim u \wedge 1 = u$ by closure under binary meets.
 Similarly, $v \wedge a \sim v$. But $u \wedge a = v \wedge a$ by assumption and so $u \sim v $ by transitivity (and symmetry). Therefore, $\Delta_a \subseteq \langle (a,1) \rangle$ and the claim follows.
\end{proof}

By considering the least element in the equivalence class of $1$ with respect to an open congruence $\Delta_a$ we can recover the element $a$ and hence the above assignment of opens to (open) sublocales is injective. Therefore, we can identify opens and open sublocales without causing confusion.

There is also a notion of closed sublocale. For each $a \in \O X$ the corresponding closed sublocale is given by setting the proposition $a$ to be \emph{false}.
\begin{definition}\label{def:closed_sublocales}
 Let $X$ be a locale and consider $a \in \O X$. The \emph{closed sublocale} corresponding to $a$ is given by the frame quotient from $\O X$ to ${\uparrow} a$ sending $u$ to $u \vee a$.
\end{definition}
As above, we can prove that the kernel $\nabla_a$ of this map is indeed given by $\langle (0,a)\rangle$.

This definition of closed sublocales also makes sense intuitively. If we think of closed `sets' as formal complements of opens, then $(\O X)\op$ can be thought of as the lattice of closed sets of $X$. With respect to this reversed order we see that closed sets of $X$ restrict to closed sublocales by taking an intersection, as we would expect.

Now let $\Sloc X$ denote the lattice of all sublocales of $X$. This is isomorphic to the lattice of congruences on $\O X$ under the \emph{reverse order}. We will not need this, but it is notable that $\Sloc X$ is a distributive lattice. In fact, $(\Sloc X)\op$  is a frame! (See \citet*[Section IV.5]{galoisTheoryGrothendieck} for details.)
The open and closed sublocales behave as we might expect inside $\Sloc X$. Most importantly, we have the following.
\begin{lemma}
 The open and closed sublocales induced by the element $a$ are mutual complements in the lattice of sublocales.\footnote{We remark that since $\Sloc X$ is a distributive lattice, the complement of a sublocale is unique when it exists.}
\end{lemma}

\begin{proof}
 We will work in the lattice of congruences.
 Firstly, note that $\nabla_a \vee \Delta_a = \langle (0,a) \rangle \vee \langle (a,1) \rangle \supseteq \langle (0,1) \rangle$, which is easily seen to be the largest congruence $\O X \times \O X$.
 
 Now take $(u,v) \in \nabla_a \cap \Delta_a$. This means $u \vee a = v \vee a$ and $u \wedge a = v \wedge a$. We wish to show $u = v$ so that this intersection is the equality relation, the least congruence on $\O X$. This is a standard result about distributive lattices, but we give the full proof for completeness.
 
 Consider $(u \wedge v) \vee (u \wedge a) = u \wedge (v \vee a)$.
 Since $(u,v) \in \nabla_a$, we know $v \vee a = u \vee a$ and so $u \wedge (v \vee a) = u$.
 Similarly, $(u \wedge v) \vee (v \wedge a) = v$.
 But also $u \wedge a = v \wedge a$ since $(u,v) \in \Delta_a$ and hence these are all equal.
\end{proof}
\begin{remark}
 If open sublocales correspond to verifiable properties, closed sublocales correspond to \emph{refutable} properties. These are properties for which we can ascertain their falsehood whenever they are false.
\end{remark}

We can also show that open sublocales are closed under finite meets and arbitrary intersections in $\Sloc X$ (and dually for closed sublocales). This follows from the following result about closed congruences. (Recall that congruences have the reverse order to sublocales.)
\begin{lemma}
 The map $a \mapsto \nabla_a$ preserves finite meets and arbitrary joins.
\end{lemma}
\begin{proof}
 The assignment is easily seen to be monotone. Thus, $\nabla_{a \wedge b} \subseteq \nabla_a \cap \nabla_b$. Conversely, suppose $(u,v) \in \nabla_a \cap \nabla_b$.
 Then $u \vee (a \wedge b) = (u \vee a) \wedge (u \vee b) = (v \vee a) \wedge (v \vee b) = v \vee (a \wedge b)$ and so $(u,v) \in \nabla_{a \wedge v}$, as required.
 
 As for join preservation, monotonicity gives $\bigvee_{a \in A} \nabla_a \subseteq \nabla_{\bigvee A}$. On the other hand, we have $(0,\bigvee A) \in \bigvee_{a \in A} \langle (0, a)\rangle$ by closure under joins and hence $\nabla_{\bigvee A} = \langle (0, \bigvee A)\rangle \subseteq \bigvee_{a \in A} \nabla_a$.
\end{proof}

\begin{corollary}\label{cor:clopen_sublocale}
 The open sublocale given by an element $a \in \O X$ is closed if and only if $a$ is complemented in $\O X$.
\end{corollary}
Thus, these `clopen' sublocales correspond to decidable properties, in the sense that we can verify whether they are true or false.

It is enlightening to consider the case of open and closed sublocales of discrete locales.
The open sublocales of a set $X$ viewed as discrete locale simply correspond to subsets of $X$.
Then by \cref{cor:clopen_sublocale} the clopen sublocales are precisely the decidable subsets.
On the other hand, if \emph{every} sublocale of the terminal locale $1$ is open, then in particular the closed complement of each open sublocale is open, so that all subsets are decidable and excluded middle holds.

General closed sublocales of discrete locales are more subtle. They can be understood as formal complements of subsets. This explains some strange definitions that sometimes appear in constructive mathematics. For example, consider the following definition from constructive algebra.
\begin{definition}\label{def:anti-ideal}
 An \emph{anti-ideal} of a commutative ring $R$ is a subset $A \subseteq R$ such that
 \begin{itemize}
  \item $0 \notin A$,
  \item if $x + y \in A$ then $x \in A$ or $y \in A$,
  \item if $xy \in A$ then $y \in A$.
 \end{itemize}
\end{definition}
Classically such a set is simply the complement of an ideal, but this is not the case constructively since negation is badly behaved.
The localic perspective clarifies things. The usual notion of ideal is an open sublocale satisfying certain conditions, while an anti-ideal is a \emph{closed} sublocale satisfying the same conditions.
When writing these conditions in terms of the complementary open sublocale (i.e.\ a subset of $R$) we recover the above definition (see \citet*[Section 1.7 and Proposition 1.3.8]{manuell2019thesis}).
One would not be surprised to see closed ideals appearing in topological algebra and so it is not actually too surprising to see them show up when studying discrete rings.

\subsubsection{Images and preimages}

Since frames are well-behaved algebraic structures, $\Frm$ is a regular category and so has (regular epi, mono) factorisations (given by the usual factorisation of maps into a surjection followed by an injection). Hence, $\Loc$ has (epi, regular mono) factorisations and regular monomorphisms in $\Loc$ are stable under pullback.
Therefore, we can use pullbacks to take preimages of sublocales along locale morphisms as shown in the following diagram.
This makes the sublocale lattice into a functor $\Sloc\colon \Loc\op \to \Pos$.
\begin{center}
 \begin{tikzpicture}[node distance=2.5cm, auto]
  \node (A) {${\Sloc f}^*(S)$};
  \node (B) [below of=A] {$X$};
  \node (C) [right of=A] {$S$};
  \node (D) [right of=B] {$Y$};
  \draw[right hook->] (A) to node [swap] {} (B);
  \draw[->] (A) to node {} (C);
  \draw[->] (B) to node [swap] {$f$} (D);
  \draw[right hook->] (C) to node {} (D);
  \begin{scope}[shift=({A})]
   \draw +(0.3,-0.6) -- +(0.6,-0.6) -- +(0.6,-0.3);
  \end{scope}
 \end{tikzpicture}
\end{center}

By standard algebraic arguments, if $C_S$ is the frame congruence associated to $S$ then the congruence associated to ${\Sloc f}^*(S)$ is given by $\langle (f^* \times f^*)(C_S)\rangle$.
Thus, preimages of open/closed sublocales are open/closed (by their descriptions in terms of generators).

On the other hand, the (epi, regular mono) factorisation can be used directly to define \emph{images} of sublocales.
\begin{center}
 \begin{tikzpicture}[node distance=2.5cm, auto]
  \node (A) {$S$};
  \node (B) [below of=A] {$X$};
  \node (C) [right of=A] {${\Sloc f}_!(S)$};
  \node (D) [right of=B] {$Y$};
  \draw[right hook->] (A) to node [swap] {} (B);
  \draw[->>] (A) to node {} (C);
  \draw[->] (B) to node [swap] {$f$} (D);
  \draw[right hook->] (C) to node {} (D);
 \end{tikzpicture}
\end{center}
By calculating the image factorisation in $\Frm$ we easily find that the congruence corresponding to the image ${\Sloc f}_!(S)$ is given by $(f^* \times f^*)^{-1}(C_S)$.

Using the similar adjunction between images and preimages of sets (and remembering that congruences are ordered in the reverse direction to sublocales) we can see that ${\Sloc f}_!$ is left adjoint to ${\Sloc f}^*$.

\subsection{Frame coproducts and suplattices}

As mentioned above, we can use frame presentations to construct arbitrary colimits. Let us look at the case of frame coproducts in more detail.
These correspond to \emph{products} of locales.

If $(X_\alpha)_{\alpha \in I}$ is a family of locales, then the frame of the product $\prod_{\alpha \in I} X_\alpha$ has a presentation with a generator $\iota_\alpha(u)$ for each $\alpha \in I$ and $u \in \O X_\alpha$ and relations making the $\iota_\alpha$ maps into frame homomorphisms. The maps $\iota_\beta\colon \O X_\beta \to \coprod_{\alpha \in I}  \O X_\alpha$ are then the coproduct injections, though going forward we will instead refer to these as the frame maps associated to the product projections $\pi_\beta\colon \prod_{\alpha \in I} X_\alpha \to X_\beta$.

Furthermore, if each $\O X_\alpha$ is itself given by a presentation $\langle G_\alpha \mid R_\alpha\rangle$, then the product locale has a presentation $\langle \bigsqcup_{\alpha \in I} G_\alpha \mid \bigsqcup_{\alpha \in I} R_\alpha\rangle$, where $\bigsqcup_{\alpha \in I} R_\alpha$ is understood in the obvious way.

Logically, we see that the theory given by the coproduct frame describes something with aspects that are modelled by each $\O X_\alpha$. Indeed, the points of product are simply families of points from each $\O X_\alpha$, as we would expect.
The opens in the $\prod_{\alpha \in I} X_\alpha$ are given by joins of finite meets of elements of the form $\pi_\alpha^*(u)$. These finite meets are analogous to the basic open rectangles of the Tychonoff topology.
From the perspective of verifiability, it is intuitive that we can only constrain finitely many indices with such a basic open, since verifying something in the product involves checking each component in turn, and we can only carry out finitely many such checks.

\subsubsection{Finite products}

We now look at the case of finite products in more detail.
Here by `finite' we mean `Bishop-finite', since this is the most natural version of finiteness to consider for non-idempotent operations.
A general Bishop-finite product can be built up from empty products and binary products.
The empty product is the terminal locale $1$, whose frame of opens is given by $\Omega$.

For a binary product $X \times Y$, we write the corresponding frame coproduct as $\O X \oplus \O Y$. It will be convenient to define the notation $u \oplus v = \pi_1^*(u) \wedge \pi_2^*(v)$ for the basic rectangles. These are basic opens for the product in the following sense.

\begin{definition}
 We say a subset $B$ of a frame $\O X$ is a \emph{base} if every element of $\O X$ is a join of elements from $B$.
 If $G$ is a generating set, then the finite meets of elements of $G$ form a base.
\end{definition}

If $f\colon Z \to X$ and $g\colon Z \to Y$ are locale maps, the map $(f,g)\colon Z \to X \times Y$ obtained from the universal property of the product is defined by $(f,g)^*(u \oplus v) = f^*(u) \wedge g^*(v)$.

We have argued that it can be useful to treat set-based structures as discrete locales (for example, in the discussion after \cref{def:anti-ideal}).
The following result is important for allowing a seamless transition between these perspectives when only finite products are involved.
\begin{proposition}\label{lem:finite_product_of_sets}
 Bishop-finite products of discrete locales are discrete and agree with the products of the underlying sets.
\end{proposition}
\begin{proof}
 Certainly the empty product $1$ is discrete. The result will then follow if we prove it for binary products. As right adjoint, the functor from $\Loc$ to $\Top$ preserves products, and so it suffices to show that the product locale in question is spatial.
 
 We must show that elements of $\Omega^X \oplus \Omega^Y$ are determined by the points they contain. A general such element is of the form $u = \bigvee_{\alpha \in I} S_\alpha \oplus T_\alpha$.
 We have $(x,y) \in u$ if and only if $(x,y) \in S_\alpha \oplus T_\alpha$ for some $\alpha \in I$ (where $(x,y)$ corresponds to a locale map $(x,y)\colon 1 \to X \times Y$ that decomposes into $x\colon 1 \to X$ and $y\colon 1 \to Y$).
 Note that $(x,y) \in S_\alpha \oplus T_\alpha = \pi_1^*(S_\alpha) \wedge \pi_2^*(T_\alpha) \iff x \in S_\alpha \land y \in T_\alpha$.
 But $x \in S_\alpha \iff \{x\} \subseteq S_\alpha$ and similarly for $T_\alpha$. Hence, $(x,y) \in u \iff \{x\} \oplus \{y\} \le u$.
 Finally, since singletons generate the frames $\Omega^X$ and $\Omega^Y$ under joins, we can conclude that $u$ is determined by the points it contains.
\end{proof}

\subsubsection{Suplattices}

We can draw an analogy to binary coproducts of commutative rings.
In that setting, one can first construct the underlying abelian group structure of the coproduct as a tensor product of the additive groups of the constituents, before defining the appropriate multiplication operation.
Something similar happens with frame coproducts if we view the joins as a kind of addition and meets as a multiplication.
The analogue of abelian groups is then played by \emph{suplattices}.

\begin{definition}
 A \emph{suplattice} is a poset which has arbitrary joins.\footnote{This is the same thing as a complete lattice, but we use the term \emph{suplattice} to emphasise that it is the structure given by the joins which is relevant for this discussion.} A \emph{suplattice homomorphism} is a join-preserving map between suplattices.
 We denote the category of suplattices and suplattice homomorphisms by $\Sup$.
\end{definition}
Like frames, suplattices are (large) algebraic structures.
\begin{lemma}
 The free suplattice of a set $G$ is given by the powerset $\Omega^G$.
\end{lemma}
\begin{proof}
 If $L$ is a suplattice and $f\colon G \to L$ is a function then we can define a suplattice homomorphism from $\Omega^G$ by sending $S \subseteq G$ to $\bigvee_{s \in S} f(s)$. It is easy to check that this is indeed a suplattice homomorphism, and in fact, the unique one making the relevant diagram commute.
\end{proof}

\begin{lemma}
 The category $\Sup$ is complete and cocomplete.
\end{lemma}
\begin{proof}
 Limits are computed as for any algebraic structure. Now note that reversing the order gives an equivalence between $\Sup$ and the category $\Inf$ of inf{}lattices (which are defined dually in terms of meets). On the other hand, there is a \emph{dual} equivalence between $\Sup$ and $\Inf$ given by leaving the objects alone and taking the right adjoints of the suplattice homomorphisms. Composing these gives an auto-duality of $\Sup$.
 Thus, we can compute colimits in $\Sup$ by taking the right adjoints of all the morphisms involved, reversing the order, and finding the limit in the usual way.
\end{proof}

In particular, we can define suplattices by generators and relations.
Also note that the adjunction between $\Sup$ and $\Set$ factors through the category of posets $\Pos$, and the free suplattice on a poset $P$ can easily be shown to be the lattice of downsets $\D P$.
We can use this to define suplattice presentations given by a \emph{poset} of generators.
The convenient aspects of the category $\Sup$ allow us to give a rather explicit construction of the suplattice described by a presentation.
\begin{proposition}\label{prop:explicit_suplattice_presentation}
 If $G$ is a poset and $R$ is a set of relations of the form $\bigvee S \le \bigvee T$, then $\langle G \textup{ poset} \mid R\rangle_\Sup$ is order-isomorphic to the subset of $\D G$ consisting of downsets $D$ such that for each relation $\bigvee S \le \bigvee T$ in $R$, we have $S \subseteq D$ whenever $T \subseteq D$.
\end{proposition}
\begin{proof}
 It will be convenient to view $R$ as indexing set with the functions $s,t\colon R \to \Omega^G$ picking out the sets $S$ and $T$ respectively for each relation.
 We now define suplattice homomorphisms $\overline{s},\overline{t}\colon \Omega^R \to \D G$ so that $\overline{s}(\{r\})$ and $\overline{t}(\{r\})$ are the downsets generated by $s(r)$ and $t(r)$, respectively.
 
 The presented suplattice can then be expressed by the coequaliser
 \begin{center}
  \begin{tikzpicture}
   \node (A) {$\Omega^R$};
   \node [right=1.0cm of A] (B) {$\D G$};
   \node [right=1.0cm of B] (C) {$\langle G \text{ poset} \mid R\rangle_\Sup$};
   \draw[transform canvas={yshift=0.5ex},->] (A) to node [above] {$\overline{t}$} (B);
   \draw[transform canvas={yshift=-0.5ex},->] (A) to node [below] {$\overline{s} \vee \overline{t}$} (B);
   \draw[->>] (B) to node [above] {$q$} (C);
  \end{tikzpicture}
 \end{center}
 where the suplattice homomorphism $\overline{s} \vee \overline{t}$ is given by the pointwise join of $\overline{s}$ and $\overline{t}$.
 
 The right adjoint $\overline{t}_*$ of $\overline{t}$ is given by
 \begin{align*}
  \overline{t}_*(D) &= \bigvee\{ U \in \Omega^R \mid \overline{t}(U) \subseteq D \} \\
                    &= \{r \in R \mid t(r) \subseteq D\}.
 \end{align*}
 Similarly, $(\overline{s} \vee \overline{t})_*(D) = \{r \in R \mid s(r) \cup t(r) \subseteq D\}$.
 Taking the equaliser of these in $\Inf$ we obtain the set of $D \in \D G$ such that $\overline{t}_*(D) = (\overline{s} \vee \overline{t})_*(D)$, which is to say those $D$ such that $t(r) \subseteq D \iff s(r) \cup t(r) \subseteq D$ for all $r \in R$. The result follows.
\end{proof}

There is a strong link between frame presentations and suplattice presentations. Just as we can present suplattices using a poset of generators, we can present frames with a $\wedge$-semilattice of generators. We then have the following theorem (see \citet*{abramsky1993quantales}).
\begin{theorem}[Coverage theorem]\label{thm:suplattice_coverage}
 Let $G$ be a $\wedge$-semilattice and let $R$ be a set of relations of the form $\bigvee S \le \bigvee T$ such that whenever $\bigvee S \le \bigvee T$ is in $R$, then so is $\bigvee_{s \in S} u \wedge s \le \bigvee_{t \in T} u \wedge t$ for each $u \in G$. Then there is an order isomorphism \[\langle G \textup{ $\wedge$-semilattice} \mid R\rangle_\Frm \cong \langle G \textup{ poset} \mid R\rangle_\Sup.\]
\end{theorem}
Note that \emph{every} frame presentation can be brought into the above form by taking formal finite meets of the generators and adding the appropriate relations.
\begin{proof}
 We first show that $\langle G \textup{ poset} \mid R\rangle_\Sup$ is a frame. It suffices to prove it is a Heyting algebra.
 We use the construction from \cref{prop:explicit_suplattice_presentation}. Recall that $\D G$ is a frame and hence a Heyting algebra.
 In there we have $g \in D \heyting E \iff {\downarrow} g \cap D \subseteq E$. We will show that if $E$ satisfies the conditions of \cref{prop:explicit_suplattice_presentation}, then so does $D \heyting E$ and hence this also gives the Heyting implication in the sub-inf{}lattice.
 
 Suppose $\bigvee S \le \bigvee T$ is a basic relation. Then by assumption, for each $d \in G$ we have a basic relation $\bigvee (S \cap {\downarrow} d) \le \bigvee (T \cap {\downarrow} d)$
 and hence $T \cap {\downarrow} d \subseteq E$ implies $S \cap {\downarrow} d \subseteq E$. Taking the join over all $d \in D$ and applying the defining adjunction of Heyting implication we see that $T \subseteq D \heyting E$ implies $S \subseteq D \heyting E$, as required.
 
 We now argue that the canonical map from $G$ to $\langle G \textup{ poset} \mid R\rangle_\Sup$ satisfies the necessary universal property for $\langle G \textup{ $\wedge$-semilattice} \mid R\rangle_\Frm$. Note that it preserves finite meets.
 Now let $L$ be a frame and let $f\colon G \to L$ be a $\wedge$-semilattice homomorphism such that $\bigvee_{s \in S} f(s) \le \bigvee_{t \in T} f(t)$ for each basic relation $\bigvee S \le \bigvee T$.
 Since $f$ is in particular monotone, there is a unique suplattice homomorphism $\overline{f}\colon \langle G \textup{ poset} \mid R\rangle_\Sup \to L$ such that $\overline{f}$ agrees with $f$ on generators. It just remains to show that $\overline{f}$ preserves finite meets. Since $f$ is a $\wedge$-semilattice homomorphism, $\overline{f}$ preserves finite meets on a base. The result then follows by taking joins.
\end{proof}

Now we would like to define tensor products of suplattices. Recall that the tensor product $A \otimes B$ of abelian groups is given by the universal bilinear map in the sense that bilinear maps from $A \times B$ to $C$ are in bijection with linear maps (i.e.\ group homomorphisms) from $A \otimes B$ to $C$, naturally in $C$.

\begin{definition}
 If $L, M$ and $N$ are suplattices, we say a function $f\colon L \times M \to N$ is \emph{bilinear} in $L$ and $M$ if $f(\bigvee_{s \in S} s, t) = \bigvee_{s\in S} f(s,t)$
 and $f(s, \bigvee_{t \in T} t) = \bigvee_{t\in T} f(s,t)$.
 
 The tensor product $L \otimes M$ is defined by the following universal property: there is a bilinear map $\beta\colon L \times M \to L \otimes M$ such that for every bilinear map $f\colon L \times M \to N$ there is a unique suplattice homomorphism $\widetilde{f}\colon L \otimes M \to N$ such that $f = \widetilde{f}\beta$.
\end{definition}

The tensor product $L \otimes M$ is easily seen to be given by the free suplattice on the generators $\ell \otimes m$ for $\ell \in L$ and $m \in M$, subject to the relations that make the map $\beta\colon (\ell,m) \mapsto \ell \otimes m$ bilinear.
\begin{remark}
We note that as a corollary of \cref{prop:explicit_suplattice_presentation}, $L \otimes M$ can be expressed as the set of downsets $D$ of $L \times M$ such that whenever $(\ell_\alpha, m) \in D$ for all $\alpha \in I$ we have $(\bigvee_{\alpha \in I} \ell_\alpha, m) \in D$ and whenever $(\ell, m_\beta) \in D$ for all $\beta \in J$ we have $(\ell, \bigvee_{\beta \in J} m_\beta) \in D$.
\end{remark}

We can now finally show that frame coproducts can be expressed in terms of suplattice tensor products.
\begin{proposition}\label{lem:frame_coproduct_as_tensor_product}
 If $X$ and $Y$ are locales, then $\O X \oplus \O Y$ is order-isomorphic to $\O X \otimes \O Y$.
\end{proposition}
\begin{proof}
 This is an immediate consequence of the coverage theorem applied to the above presentation.
\end{proof}

The tensor product can be extended to a functor ${\otimes}\colon \Sup \times \Sup \to \Sup$ by defining $f \otimes g\colon \ell \otimes m \mapsto f(\ell) \otimes g(m)$.
We also note that the lattice of truth values acts as a unit with respect to tensor product. There is a canonical isomorphism $\lambda\colon \Omega \otimes L \cong L$ defined by $\lambda\colon p \otimes \ell \mapsto \bigvee\{\ell \mid p\}$, and the inverse is given by $\ell \mapsto \top \otimes \ell$.

The tensor product functor $\otimes$ and the unit $\Omega$, together with the canonical unit, associativity and symmetry isomorphisms, give $\Sup$ the structure of a symmetric monoidal category.
In fact, $L \otimes (-)$ is left adjoint to the functor $\hom(L, (-))$, which sends a suplattice $M$ to the set $\Hom(L,M)$ of suplattice homomorphisms from $L$ to $M$ equipped with the suplattice structure coming from the pointwise ordering.
This makes $\Sup$ into a symmetric monoidal closed category.

We now obtain a very easy alternative proof of \cref{lem:finite_product_of_sets}: we have $\Omega^{X} \otimes \Omega^{Y} \cong (\coprod_{x \in X} \Omega) \otimes (\coprod_{y \in Y} \Omega) \cong (\coprod_{(x,y) \in X \times Y} \Omega) \cong \Omega^{X \times Y}$, since the left adjoint $L \otimes (-)$ preserves coproducts.

\section{Topological notions}

We are now in a position to discuss standard topological concepts such as the Hausdorff separation axiom and compactness. We will also see some other topological properties that are most relevant in the constructive setting.

\subsection{Diagonal conditions}

Recall that classically a space $X$ is Hausdorff if and only if the diagonal of $X \times X$ is closed.
We will define Hausdorff locales similarly, but let us first briefly consider the dual condition of having an \emph{open} diagonal.

\begin{lemma}\label{cor:open_diagonal_in_discrete}
 The diagonal of a discrete locale is open.
\end{lemma}
\begin{proof}
 By \cref{lem:finite_product_of_sets} the diagonal of a discrete locale is computed as in $\Set$ and hence given by subset, which is, of course, an open set in the discrete topology.
\end{proof}

So in a discrete locale equality is verifiable.
This is not quite a characterisation of discrete locales, since for instance, it is preserved by arbitrary sublocales.
We will give the precise characterisation in \cref{prop:discreteness_characterisation}.

A Hausdorff locale has \emph{refutable} equality --- that is, verifiable \emph{inequality}.
\begin{definition}
 We say a locale $X$ is \emph{Hausdorff} if the diagonal of $X \times X$ is closed.
\end{definition}

Note that a set is Hausdorff as a discrete locale precisely when it has decidable equality.
So while decidable equality is not something we worry about classically, it is a manifestation of a familiar topological concept.

The diagonal map $\Delta\colon X \to X \times X$ (obtained from applying the universal property of the product to $\id_X$ and $\id_X$) is given by the frame homomorphism $u \oplus v \mapsto u \wedge v$.
This is clearly surjective, and it being closed means that $\Delta^*(a) = \Delta^*(b) \iff a \vee d = b \vee d$ for some $d \in \O X \oplus \O X$. Of course, here $d$ is the open complement of the diagonal.
The following technical lemma can be helpful for proving Hausdorffness.
\begin{lemma}
 A locale $X$ is Hausdorff if and only if for all $u, v \in \O X$ we have $u \oplus v \le (u \wedge v) \oplus (u \wedge v) \vee d$ for some $d \in \O X \oplus \O X$ such that $\Delta^*(d) = 0$.
\end{lemma}
\begin{proof}
 Suppose $X$ is Hausdorff and the diagonal sublocale is represented by the congruence $\nabla_d$ for some $d \in \O X \oplus \O X$.
 Then $(d, 0) \in \nabla_d$, so that $\Delta^*(d) = 0$. Now since $\Delta^*(u \oplus v) = u \wedge v = (u \wedge v) \wedge (u \wedge v) = \Delta^*((u \wedge v) \oplus (u \wedge v))$, we know $u \oplus v \le u \oplus v \vee d = (u \wedge v) \oplus (u \wedge v) \vee d$, as required.
 
 On the other hand, suppose we have the above inequality. We must show $\Delta^*(a) = \Delta^*(b) \iff a \vee d = b \vee d$. The reverse implication is immediate after applying $\Delta^*$ and using $\Delta^*(d) = 0$. For the forward implication, observe that if $u \wedge v \le \Delta^*(b)$, then $u \oplus v \le (u \wedge v) \oplus (u \wedge v) \vee d \le \Delta^*(b) \oplus \Delta^*(b) \vee d$. Expressing $b$ as $\bigvee_{\alpha \in I} u'_\alpha \oplus v'_\alpha$ we have $\Delta^*(b) = \bigvee_{\alpha \in I} u'_\alpha \wedge v'_\alpha$ and hence $\Delta^*(b) \oplus \Delta^*(b) = \bigvee_{\alpha,\beta \in I} (u'_\alpha \wedge v'_\alpha) \oplus (u'_\beta \wedge v'_\beta) \le \bigvee_{\alpha,\beta \in I} (u'_\alpha \wedge v'_\alpha \wedge u'_\beta \wedge v'_\beta) \oplus (u'_\alpha \wedge v'_\alpha \wedge u'_\beta \wedge v'_\beta) \vee d \le b \vee d$,
 where the first inequality is by applying the hypothesis to each pair of $u'_\alpha \wedge v'_\alpha$ and $u'_\beta \wedge v'_\beta$.
 So $u \oplus v \le b \vee d$. Now taking the join over all $u \oplus v \le a$, we find that whenever $\Delta^*(a) \le \Delta^*(b)$ we have $a \le b \vee d$. So if $\Delta^*(a) = \Delta^*(b)$ then $a \le b \vee d$ and $b \le a \vee d$.
 Thus, $a \vee d = b \vee d$ and we have shown that $X$ is Hausdorff.
\end{proof}
\begin{remark}
 In fact, it suffices to only check the above for all $u,v$ in some generating set for $\O X$, since then it follows for general elements by taking finite meets and arbitrary joins.
\end{remark}

A very important example of a Hausdorff locale is given by the locale $\R$ of reals.
\begin{proposition}
 The locale $\R$ is Hausdorff.
\end{proposition}
\begin{proof}
 Recall that $\O \R$ has the following presentation.
 \begin{align*}
  \O \R = \langle &\llround q, \infty \rrround, \llround -\infty, q \rrround,\ q \in \Q \mid {} \\
                  &\llround p, \infty \rrround = \bigvee_{q > p} \llround q, \infty \rrround,\ \llround -\infty, q \rrround = \bigvee_{p < q} \llround -\infty, p \rrround, \\
                  & \bigvee_{q \in \Q} \llround q, \infty \rrround = 1,\ \bigvee_{q \in \Q} \llround -\infty, q \rrround = 1, \\
                  & \llround -\infty, q \rrround \wedge \llround p, \infty \rrround = 0\, \textnormal{ for $p \ge q$},\\
                  & \llround -\infty, q \rrround \vee \llround p, \infty \rrround = 1\, \textnormal{ for $p < q$} \rangle
 \end{align*}
 For convenience we define $\llround p, q \rrround = \llround p, \infty\rrround \wedge \llround -\infty, q\rrround$.
 
 The putative diagonal complement is \[d = \bigvee_{r \in \Q} \llround -\infty, r \rrround \oplus \llround r, \infty \rrround \vee \bigvee_{r \in \Q}\, \llround r, \infty \rrround \oplus \llround -\infty, r \rrround.\]
 To show Hausdorffness we have the following cases to prove:
 \begin{itemize}
  \item $\llround p, \infty\rrround \oplus \llround q, \infty\rrround \le \llround \max(p,q), \infty\rrround \oplus \llround \max(p,q), \infty\rrround \vee d$,
  \item $\llround -\infty, p \rrround \oplus \llround -\infty, p \rrround \le \llround -\infty, \min(p,q) \rrround \oplus \llround -\infty, \min(p,q) \rrround \vee d$,
  \item $\llround p, \infty\rrround \oplus \llround -\infty, q\rrround \le \llround p, q \rrround \oplus \llround p, q \rrround \vee d$,
  \item $\llround -\infty, q\rrround \oplus \llround p, \infty\rrround \le \llround p, q \rrround \oplus \llround p, q \rrround \vee d$.
 \end{itemize}
 We will show the third one as a representative example. The other cases are similar.
 Note that if $p \ge q$ then trivially $\llround p, \infty\rrround \oplus \llround -\infty, q\rrround \le \llround q, \infty\rrround \oplus \llround -\infty, q\rrround \le d$, and so we may assume $p < q$.
 
 The following illustration suggests we try show $\llround p, \infty\rrround \oplus \llround -\infty, q\rrround \le \llround p, q \rrround \oplus \llround p, q \rrround \vee \llround p, \infty \rrround \oplus \llround -\infty, p \rrround \vee \llround q, \infty \rrround \oplus \llround -\infty, q \rrround$.
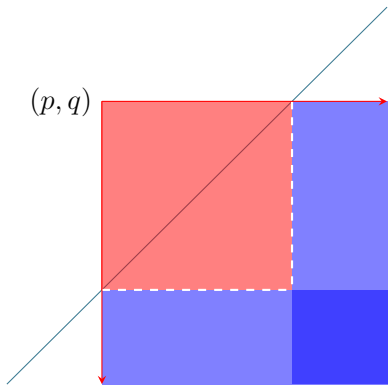
 \begin{figure}[H]
 \begin{tikzpicture}[scale=1.25]
 \draw[color=cyan!50!black] (0,0) to (4,4);

  \begin{pgfonlayer}{over}
  \draw[color=red, ->] (1,3) node[left,text=black, scale=0.9] {$(p,q)$} to (1,0);
  \draw[color=red, ->] (1,3) to (4,3);
  \end{pgfonlayer}
  \draw[draw=none,fill=red,opacity=0.5] (1,3) rectangle (3,1);
  
  \draw[draw=none,fill=blue,opacity=0.5] (1,1) rectangle (4,0);
  
  \draw[draw=none,fill=blue,opacity=0.5] (3,3) rectangle (4,0);
  
  \draw[dashed,thick,white] (1,1) to (3,1);
  \draw[dashed,thick,white] (3,3) to (3,1);
 \end{tikzpicture}
 \end{figure}
 However, this is not quite right since the white dotted lines are missing from the right-hand side.
 But we can fix this if we shrink the left-hand side by a small amount to allow for an overlap and instead try showing $\llround p+\epsilon, \infty\rrround \oplus \llround -\infty, q-\epsilon\rrround \le \llround p, q \rrround \oplus \llround p, q \rrround \vee \llround p+\epsilon, \infty \rrround \oplus \llround -\infty, p+\epsilon \rrround \vee \llround q-\epsilon, \infty \rrround \oplus \llround -\infty, q-\epsilon \rrround$ as in the illustration below.
 
 
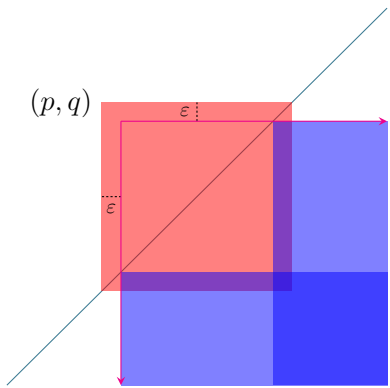
\begin{figure}[H]
 \begin{tikzpicture}[scale=1.25]
  \draw[color=cyan!50!black] (0,0) to (4,4);
  \begin{pgfonlayer}{over}
  \draw[color=magenta, ->] (1.2,2.8) to (1.2,0);
  \draw[color=magenta, ->] (1.2,2.8) to (4,2.8);
  \end{pgfonlayer}
  \draw[draw=none,fill=red,opacity=0.5] (1,3) node[left,text=black,text opacity=1,scale=0.9] {$(p,q)$} rectangle (3,1);
  
  \tikzstyle{dotted}= [dash pattern=on 2\pgflinewidth off 2\pgflinewidth]
  \draw[dotted] (1,2) to node[midway,below,scale=0.66] {$\epsilon$} (1.2,2);
  \draw[dotted] (2,2.8) to node[midway,left,scale=0.66] {$\epsilon$} (2,3);
  
  \draw[draw=none,fill=blue,opacity=0.5] (1.2,1.2) rectangle (4,0);
  \draw[draw=none,fill=blue,opacity=0.5] (2.8,2.8) rectangle (4,0);
 \end{tikzpicture}
 \end{figure}
 Let us prove this holds.
 First consider the join of the smaller red square and the part of the blue region below it:
 \begin{align*}
   &\phantom{{}={}} \llround p+\epsilon, q \rrround \oplus \llround p, q-\epsilon \rrround \vee \llround p+\epsilon, q \rrround \oplus \llround -\infty, p+\epsilon \rrround \\
                 &= \llround p+\epsilon, q \rrround \oplus \left[ \llround p, q-\epsilon \rrround \vee \llround -\infty, p+\epsilon \rrround \right] \\
                 &= \llround p+\epsilon, q \rrround \oplus \left[ \llround p, \infty \rrround \wedge \llround -\infty, q-\epsilon\rrround \vee \llround -\infty, p+\epsilon \rrround \right] \\
                 &= \llround p+\epsilon, q \rrround \oplus \left[ \left( \llround p, \infty \rrround \vee \llround -\infty, p+\epsilon \rrround \right) \wedge \llround -\infty, q-\epsilon\rrround \right] \\
                 &= \llround p+\epsilon, q \rrround \oplus \llround -\infty, q-\epsilon\rrround.
 \end{align*}
 
 Now we connect this resulting region to the blue region beside it:
 \begin{align*}
   &\phantom{{}={}} \llround p+\epsilon, q \rrround \oplus \llround -\infty, q-\epsilon\rrround \vee \llround q-\epsilon, \infty \rrround \oplus \llround -\infty, q-\epsilon \rrround \\
                 &= \left[ \llround p+\epsilon, q \rrround \vee \llround q-\epsilon, \infty \rrround \right] \oplus \llround -\infty, q-\epsilon\rrround \\
                 &= \left[ \llround p+\epsilon, \infty \rrround \wedge \llround -\infty, q \rrround \vee \llround q-\epsilon, \infty \rrround \right] \oplus \llround -\infty, q-\epsilon\rrround \\
                 &= \left[ \llround p+\epsilon, \infty \rrround \wedge \left( \llround -\infty, q \rrround \vee \llround q-\epsilon, \infty \rrround \right) \right] \oplus \llround -\infty, q-\epsilon\rrround \\
                 &= \llround p+\epsilon, \infty \rrround \oplus \llround -\infty, q-\epsilon\rrround.
 \end{align*}
 
 In summary, we have the join ${\llround p+\epsilon, q \rrround} \oplus {\llround p, q-\epsilon \rrround} \vee {\llround p+\epsilon, q \rrround} \oplus {\llround -\infty, p+\epsilon \rrround} \vee {\llround q-\epsilon, \infty \rrround} \oplus {\llround -\infty, q-\epsilon \rrround} = \llround p+\epsilon, q \rrround \oplus \llround -\infty, q-\epsilon\rrround \vee \llround q-\epsilon, \infty \rrround \oplus \llround -\infty, q-\epsilon \rrround = \llround p+\epsilon, \infty \rrround \oplus \llround -\infty, q-\epsilon\rrround$.
 Thus, $\llround p+\epsilon, \infty\rrround \oplus \llround -\infty, q-\epsilon\rrround \le \llround p, q \rrround \oplus \llround p, q \rrround \vee d$.
 
 But now taking the join over all sufficiently small $\epsilon > 0$, we have
 \begin{align*}
  \llround p, q \rrround \oplus \llround p, q \rrround \vee d &\ge \bigvee\nolimits_{\!\epsilon\,} \llround p+\epsilon, \infty\rrround \oplus \llround -\infty, q-\epsilon\rrround \\
   &\ge \bigvee\nolimits_{\!\epsilon',\epsilon''\,} \llround p+\epsilon', \infty\rrround \oplus \llround -\infty, q-\epsilon''\rrround \\
   &= \bigvee\nolimits_{\!\epsilon'} \pi_1^*( \llround p+\epsilon', \infty\rrround ) \wedge \bigvee\nolimits_{\!\epsilon''} \pi_2^*( \llround -\infty, q-\epsilon''\rrround ) \\
   & {} = \llround p, \infty\rrround \oplus \llround -\infty, q\rrround,
 \end{align*}
 and this is precisely what we wanted to show.
\end{proof}

Hausdorffness for locales satisfies all the familiar closure properties.
\begin{lemma}
 Sublocales of Hausdorff locales are Hausdorff.
\end{lemma}
\begin{proof}
 Suppose $i\colon S \hookrightarrow X$ is a sublocale inclusion.
 It is not difficult to show the following commutative square is a pullback.
 \begin{center}
  \begin{tikzpicture}[node distance=2.5cm, auto]
   \node (A) {$S$};
   \node (B) [below of=A] {$X$};
   \node (C) [right of=A] {$S \times S$};
   \node (D) [right of=B] {$X \times X$};
   \draw[right hook->] (A) to node [swap] {$i$} (B);
   \draw[right hook->] (A) to node {$(\id, \id)$} (C);
   \draw[right hook->] (B) to node [swap] {$(\id, \id)$} (D);
   \draw[right hook->] (C) to node {$i \times i$} (D);
   \begin{scope}[shift=({A})]
     \draw +(0.3,-0.6) -- +(0.6,-0.6) -- +(0.6,-0.3);
   \end{scope}
  \end{tikzpicture}
 \end{center}
 So the diagonal map $(\id,\id) \colon S \hookrightarrow S \times S$ is the restriction of the diagonal of $X$ along the sublocale inclusion $i \times i$ and is therefore closed whenever the diagonal of $X$ is closed.
\end{proof}

\begin{proposition}
 Products of Hausdorff locales are Hausdorff.
\end{proposition}
\begin{proof}
 Let $(X_\alpha)_{\alpha \in I}$ be a family of Hausdorff locales.
 For each $\beta \in I$ we can form the pullback.
 \begin{center}
  \begin{tikzpicture}[node distance=2.5cm, auto]
   \node (A) {$E_\beta$};
   \node (B) [below of=A] {$X_\beta$};
   \node (C) [right of=A] {$(\prod_{\alpha \in I} X_\alpha)^2$};
   \node (D) [below of=C] {$X_\beta \times X_\beta$};
   \draw[->] (A) to node [swap] {$p_\beta$} (B);
   \draw[right hook->] (A) to node {$e_\beta$} (C);
   \draw[right hook->] (B) to node [swap] {$(\id, \id)$} (D);
   \draw[->] (C) to node {$(\pi_\beta \pi_1, \pi_\beta \pi_2)$} (D);
   \begin{scope}[shift=({A})]
     \draw +(0.3,-0.6) -- +(0.6,-0.6) -- +(0.6,-0.3);
   \end{scope}
  \end{tikzpicture}
 \end{center}
 Here $E_\beta$ is intuitively the closed sublocale of $(\prod_{\alpha \in I} X_\alpha) \times (\prod_{\alpha \in I} X_\alpha)$ where the two $\beta$ coordinates are equal. (We do not try to express this as $X_\beta \times \prod_{\alpha \ne \beta} X_\alpha$, since $I$ might not have decidable equality). The sublocale is closed, since it is the preimage of the closed sublocale given by the diagonal of $X_\beta$.
 
 Now we claim the diagonal $(\id,\id)\colon \prod_{\alpha \in I} X_\alpha \to (\prod_{\alpha \in I} X_\alpha) \times (\prod_{\alpha \in I} X_\alpha)$ is the intersection of all the sublocales $E_\beta$.
 From this diagonal map $(\id, \id)$ and the projection $\pi_\beta\colon \prod_{\alpha \in I} X_\alpha \to X_\beta$ we obtain a map $d_\beta\colon \prod_{\alpha \in I} X_\alpha \to E_\beta$ for each $\beta \in I$ from the universal property of the pullback. We want to show these maps make $\prod_{\alpha \in I} X_\alpha$ the wide pullback of the $e_\beta$ maps.
 
 Suppose we have maps $h_\beta \colon Y \to E_\beta$ such that $e_\beta h_\beta$ is the same for each $\beta \in I$. Call this common composite $(f,g)\colon Y \to (\prod_{\alpha \in I} X_\alpha)^2$.
 Then by commutativity of the diagram above we have $(\pi_\beta \pi_1 f, \pi_\beta \pi_2 g) = (\pi_\beta \pi_1, \pi_\beta p_2) \circ e_\beta h_\beta = (\id,\id) \circ p_\beta h_\beta = (p_\beta h_\beta, p_\beta h_\beta)$. So the composites of $f$ and $g$ with each projection are equal and hence $f = g$.
 Now from $(\id, \id) \circ f = (f,f) = (f,g)$ it follows that $f\colon Y \to \prod_{\alpha \in I} X_\alpha$ is a map making the following diagram commute.
 It is unique since $(\id,\id)\colon \prod_{\alpha \in I} X_\alpha \to (\prod_{\alpha \in I} X_\alpha)^2$ is monic.
 
 \begin{center}
  \begin{tikzpicture}[node distance=2.5cm, auto]
   \node (A) {$\prod_{\alpha \in I} X_\alpha$};
   \node (B) [below of=A] {$E_{\beta'}$};
   \node (C) [right of=A] {$E_\beta$};
   \node (D) [below of=C] {$(\prod_{\alpha \in I} X_\alpha)^2$};
   \draw[right hook->] (A) to node [swap] {$d_{\beta'}$} (B);
   \draw[right hook->] (A) to node {$d_\beta$} (C);
   \draw[right hook->] (B) to node [swap] {$e_{\beta'}$} (D);
   \draw[right hook->] (C) to node {$e_\beta$} (D);
   
   \path (B) -- node[auto=false]{\rotatebox{45}{\large\ldots}} (C);
   
   \node (X) [above left=1.2cm and 1.2cm of A.center] {$Y$};
   \draw[out=-90,->] (X) to node [swap] {$h_{\beta'}$} (B);
   \draw[out=0,->] (X) to node {$h_\beta$} (C);
   \draw[->, pos=0.55, dashed] (X) to node {$f$} (A);
  \end{tikzpicture}
 \end{center}
 
 Thus, we have shown the diagonal is indeed the intersection of the $E_\beta$ sublocales. Now since the $E_\beta$ sublocales are closed and closed sublocales are closed under arbitrary intersections, we find that the diagonal of $\prod_{\alpha \in I} X_\alpha$ is closed, as required.
\end{proof}

Note that the previous two results are rather intuitive from the verifiability perspective: if we can refute equality in a larger space we can surely also do so in a subspace and to verify that two tuples in a product are different, we need only verify that they differ at some index.

\subsection{Compactness and overtness}\label{subsec:compactness_and_overtness}

Compactness is a notion of central importance in classical topology, but it can take some time before students understand its significance. The intuition of compact spaces being ``like finite spaces'' in some sense is good, but too vague. Yet again, the perspective of verifiability provides a very satisfactory answer: compact spaces are those which verifiable properties can be universally quantified over.

Given a verifiable predicate $\phi$ on the product $X \times Y$, when does $\forall x \in X.\ \phi(x,y)$ give a verifiable predicate on $Y$?
This is easy if $X$ is (Kuratowski-)finite --- just check that $\phi(x,y)$ holds for each element of $X$ in turn. For an infinite set this is impossible to verify in finite time, but 
it may come as a surprise that there are locales with infinitely many points for which this is indeed possible.\footnote{For an interpretation of this idea in a slightly different setting see \citet*[Section 3.11]{escardo2004synthetic}, which gives an algorithm for universal quantification over Cantor space.}

To describe what we mean by universal quantification formally without reference to points it is best to use hyperdoctrines.\footnote{See \citet*{pittsLogic} for a good introduction to categorical logic and \citet*{manuell2019thesis} for some applications to topology. We can formally relate compactness to universal quantification if we use the hyperdoctrine of open sublocales, and dually we can relate it to \emph{existential} quantification if we use the hyperdoctrine of closed sublocales. We note that in the former case we should assume an additional Frobenius reciprocity condition for universal quantification that is usually omitted outside of classical logic.}
However, in these notes we would like to avoid the apparatuses of categorical logic and so we will make do with an informal justification using sets of points and classical logic.

Suppose $U$ is an open subset of $X \times Y$. We want to know if the set $\{y \in Y \mid \forall x \in X.\ (x,y) \in U\}$ is open. Taking complements, this is the same as the set $\{y \in Y \mid \exists x \in X.\ (x,y) \notin U\}$ being closed. But this is just the image of the closed set $U\comp$ under the projection $\pi_2\colon X \times Y \to Y$.
Thus, at least intuitively, we have argued that the universal quantification over $X$ of a verifiable predicate is always verifiable precisely if the images of closed sublocales under $\pi_2\colon X \times Y \to Y$ are always closed for all $Y$.

We will see later that this indeed characterises when $X$ is compact.
However, first let us consider the dual concept: what are the locales which we can \emph{existentially} quantify verifiable propositions over?
Analogously to the compact case we can argue these are the locales $X$ such that the projections $\pi_2\colon X \times Y \to Y$ always send open sublocales to open sublocales.
We call such locales \emph{overt}.

You are probably wondering why, if compactness is such an important concept, you have not encountered the dual concept of overtness before.
This is because classically \emph{every} locale is overt! However, constructively it is nontrivial, and indeed, just as important as compactness.
Overtness is technically a little easier to deal with than compactness, so we will take a look at it first.

\subsubsection{Overtness}

Before we can give a formal definition of overtness we must first define the notion of an open map.
\begin{definition}
 A locale map $f\colon X \to Y$ is \emph{open} if the image map ${\Sloc f}_!\colon \Sloc X \to \Sloc Y$ sends open sublocales to open sublocales.
\end{definition}

\begin{remark}
 Note that a sublocale inclusion is open if and only if it is an open sublocale as defined in \cref{def:open_sublocales}.
\end{remark}

It will be convenient to have an alternative description of open maps.
\begin{lemma}\label{lem:open_map_characterisation}
 A locale map $f\colon X \to Y$ is open if and only if $f^*\colon \O Y \to \O X$ has a left adjoint $f_!\colon \O X \to \O Y$ satisfying $f_!(f^*(b) \wedge a) = b \wedge f_!(a)$.
\end{lemma}
\begin{proof}
 If $f$ is open then by the definition of ${\Sloc f}_!$ we have $(f^* \times f^*)^{-1}(\Delta_a) = \Delta_{g(a)}$ for some map $g\colon \O X \to \O Y$.
 This means $f^*(u) \wedge a = f^*(v) \wedge a \iff u \wedge g(a) = v \wedge g(a)$.
 
 Now taking $u = 1$ we find $a \le f^*(v) \iff g(a) \le v$ and so $g$ is left adjoint to $f^*$.
 On the other hand, letting $v = w \wedge u$, the right-hand side becomes $u \wedge g(a) \le w$
 and the left-hand side becomes $f^*(u) \wedge a \le f^*(w)$,
 which is equivalent to $f_!(f^*(u) \wedge a) \le w$.
 Thus, $f_!(f^*(u) \wedge a) = u \wedge g(a) = u \wedge f_!(a)$.
 
 Conversely, suppose $f_!(f^*(b) \wedge a) = b \wedge f_!(a)$ for all $a \in \O X$ and $b \in \O Y$. Then $f^*(b) \wedge a \le f^*(w) \iff b \wedge f_!(a) \le w$.

 Now $f^*(b) \wedge a = f^*(w) \wedge a$ implies $f^*(b) \wedge a \le f^*(w)$ and $f^*(w) \wedge a \le f^*(b)$, yielding $b \wedge f_!(a) \le w$ and $w \wedge f_!(a) \le b$ by the above. Hence, it entails $b \wedge f_!(a) = w \wedge f_!(a)$.
 On the other hand, $b \wedge f_!(a) = w \wedge f_!(a)$ implies $b \wedge f_!(a) \le w$ and $w \wedge f_!(a) \le b$ and hence $f^*(b) \wedge a \le f^*(w)$ and $f^*(w) \wedge a \le f^*(b)$, giving $f^*(b) \wedge a = f^*(w) \wedge a$.
 
 Thus, we have shown $f^*(b) \wedge a = f^*(w) \wedge a \iff b \wedge f_!(a) = w \wedge f_!(a)$,
 which is precisely what it means to have $(f^* \times f^*)^{-1}(\Delta_a) = \Delta_{f_!(a)}$ and so $f$ is open.
\end{proof}

We can now define overtness.

\begin{definition}\label{def:overt}
 A locale $X$ is called \emph{overt} if the unique map ${!}\colon X \to 1$ is open.
 
 Some people call these \emph{open} locales, but to avoid confusion we will not.
\end{definition}
\begin{remark}
 We argued above that a locale should be overt if and only if $\pi_2\colon X \times Y \to Y$ is open for all locales $Y$. \Cref{def:overt} only asserts this for $Y = 1$, but it actually follows for all $Y$, since we will see later (in \cref{thm:open_pullback_stable}) that open maps are stable under pullback.
\end{remark}

\begin{lemma}
 A locale $X$ is overt if and only if the unique frame map ${!}^*\colon \Omega \to \O X$ has a left adjoint $\exists\colon \O X \to \Omega$.
\end{lemma}
\begin{proof}
 If $X$ is overt then ${!}^*$ has a left adjoint by \cref{lem:open_map_characterisation}. Again by \cref{lem:open_map_characterisation}, for the other direction we need to show that $\exists({!}^*(b) \wedge a) = b \wedge \exists(a)$.
 But ${!}^*(b) \wedge a = \bigvee \{a \mid b = \top\}$ and this join is preserved by the left adjoint $\exists$. Hence, the result follows.
\end{proof}

Let us look more closely at the map $\exists$.
Since it maps into $\Omega$ it is determined by what it sends to $\top$.
Suppose $\exists(u) = \top$. Then if $u \le \bigvee S$ we have $u \le \bigvee \{s \mid s \in S\} \le \bigvee\{1 \mid s \in S\} = {!}^*(\llbracket \exists s \in S\rrbracket)$,
so that $\top = \exists(u) \le \llbracket \exists s \in S\rrbracket$ and hence $S$ is inhabited. We say that such an element $u$ is \emph{positive}.
\begin{definition}
 An open $u$ in a locale $X$ is said to be \emph{positive} (written $u > 0$) if $u \le \bigvee S$ implies $S$ is inhabited.
 We say $X$ is positive if $1 \in \O X$ is positive.
\end{definition}
This is a more useful way to formulate a frame element being nonzero constructively than simply saying $u \ne 0$.
It is easy to see that an open of a discrete locale is positive if and only if the corresponding subset is inhabited.
More generally, any open that contains a point is positive, but the converse does not hold.

\begin{lemma}
 A locale $X$ is positive if and only if the unique map ${!}\colon X \to 1$ is epic.
\end{lemma}
\begin{proof}
 Suppose $X$ is positive. We must show ${!}^*\colon \Omega \to \O X$ is injective.
 Suppose $\bigvee\{1 \mid p = \top\} \le \bigvee\{1 \mid q = \top\}$. Then if $p = \top$, we have $\bigvee\{1 \mid q = \top\} \ge 1 > 0$ and hence $q = \top$. Thus, ${!}^*$ reflects order and is therefore injective.
 
 Now suppose ${!}^*$ is injective and consider $\bigvee S = 1 \in \O X$. We must show $S$ is inhabited. Note that $1 \le \bigvee\{s \mid s \in S\} \le \{1 \mid s \in S\} = {!}^*(\llbracket \exists s \in S\rrbracket)$. But then ${!}^*(\llbracket \exists s \in S\rrbracket) = {!}^*(\top)$ and so $\llbracket \exists s \in S\rrbracket = \top$ by injectivity.
\end{proof}
\begin{corollary}\label{cor:image_and_preimage_of_positive_locales}
 If $f\colon X \to Y$ is a map of locales and $X$ is positive, then so is $Y$. The converse holds if $f$ is epic.
\end{corollary}
\begin{proof}
 This follows from simple properties of composites of epimorphisms.
\end{proof}

As suggested above the map $\exists$ measures the positivity of opens.
\begin{lemma}
 In an overt locale $\exists(u) = \llbracket u > 0 \rrbracket$.
\end{lemma}
\begin{proof}
 We have already shown $\exists(u) = \top \implies u > 0$.
 For the converse, suppose $u > 0$. Then $u \le {!}^*\exists(u) = \bigvee\{1 \mid \exists(u) = \top\}$ and hence $\exists(u) = \top$ by positivity.
\end{proof}
\begin{remark}
 Note that if an open $u$ in an overt locale $X$ is \emph{not} positive, then $\exists(u) \le \bot$ and hence $u \le {!}^*(\bot) = 0$.
\end{remark}

We can characterise overt locales in terms of having sufficiently many positive elements. For this reason overt locales are sometimes called \emph{locally positive} locales.
\begin{proposition}\label{prop:locally_positive}
 A locale $X$ is overt if and only if $\O X$ has a base of positive elements.
\end{proposition}
\begin{proof}
 Suppose $X$ is overt. We will show the set of all positive elements forms a base. For $u \in \O X$ we have $u \le {!}^*\exists(u) = \bigvee\{1 \mid u > 0\}$.
 Meeting with $u$ then gives $u = \bigvee\{u \mid u > 0\}$, which is indeed a join of positive elements.
 
 Now suppose $B$ is a base of positive elements for $\O X$. We claim the map $\exists\colon u \mapsto \llbracket u > 0\rrbracket$ is left adjoint to ${!}^*$.
 It is clear that if $\bigvee\{1 \mid p = \top\} > 0$ then $p = \top$ and so $\exists \circ {!}^* \le \id_\Omega$.
 Now take $a \in \O X$. We have $a = \bigvee\{b \in B \mid b \le a\} \le \bigvee\{1 \mid b \in B,\, b \le a\} = {!}^*(\llbracket \exists b \in B.\ b \le a \rrbracket)$.
 But if $a \ge b > 0$ then $a > 0$ and so $a \le {!}^*(\llbracket \exists b \in B.\ b \le a \rrbracket) \le {!}^*\exists(a)$. Thus, $\id_{\O X} \le {!}^*\circ \exists$ and $\exists$ is left adjoint to ${!}^*$.
\end{proof}
\begin{remark}
 If $X$ is overt we can replace any join by one of consisting only of positive elements, since as above $\bigvee A = \bigvee\{\bigvee\{a \mid a > 0\} \mid a \in A\} = \bigvee\{a \in A \mid a > 0 \}$.
\end{remark}

It is now clear that discrete locales are overt, since the singletons form a base of positive elements.
Moreover, classically the set of nonzero elements always form a base and so \emph{every} locale is overt!

\begin{lemma}\label{lem:open_subloc_of_overt}
 Open sublocales of overt locales are overt.
\end{lemma}
\begin{proof}
 Observe that the composition of open maps is open. Thus, if ${!}\colon X \to 1$ is open and $U \hookrightarrow X$ is open, then so is the composite map from $U$ to $1$.
\end{proof}
\begin{remark}
 This result is in a sense dual to the more familiar result that closed subspaces of compact spaces are compact (see \cref{lem:closed_subloc_of_compact}).
 It is also intuitively clear from the point of view of verifiability, since we can verify $\exists x \in U.\ \phi(x)$ by verifying $\exists x \in X.\ x \in U \land \phi(x)$.
\end{remark}
\begin{lemma}\label{lem:image_of_overt_subloc}
 Images of overt (sub)locales are overt.
\end{lemma}
\begin{proof}
 It suffices to show subframes of overt frames are overt.
 Let $e^*\colon \O Y \hookrightarrow \O X$ be a subframe inclusion and suppose ${!}^*\colon \Omega \to \O X$ has a left adjoint $\exists\colon \O X \to \Omega$.
 We claim that $\exists \circ e^*$ is the left adjoint of ${!}^*\colon \O Y \to \Omega$.
 Indeed, observe that $\exists(e^*(u)) \le p \iff e^*(u) \le {!}^*(p) = e^*({!}^*(p)) \iff u \le {!}^*(p)$, since $e^*$ is an order embedding.
\end{proof}

The following theorem is an important result about open maps.
\begin{theorem}\label{thm:open_pullback_stable}
 Open maps are stable under pullback.
\end{theorem}
\begin{proof}
 Consider the following pullback diagram in $\Loc$.
 \begin{center}
  \begin{tikzpicture}[node distance=2.5cm, auto]
    \node (A) {$X\times_Z Y$};
    \node (B) [below of=A] {$X$};
    \node (C) [right of=A] {$Y$};
    \node (D) [right of=B] {$Z$};
    \draw[->] (A) to node [swap] {$g'$} (B);
    \draw[->] (A) to node {$f'$} (C);
    \draw[->] (B) to node [swap] {$f$} (D);
    \draw[->] (C) to node {$g$} (D);
    \begin{scope}[shift=({A})]
        \draw +(0.3,-0.6) -- +(0.6,-0.6) -- +(0.6,-0.3);
    \end{scope}
  \end{tikzpicture}
 \end{center}
 The frame $\O (X\times_Z Y)$ can be found as a quotient of the frame coproduct $\O X \oplus \O Y$ by the congruence generated by $\pi_1^* f^*(w) \sim \pi_2^* g^*(w)$ for $w \in \O Z$.
 As in \cref{lem:frame_coproduct_as_tensor_product} we can show that $\O (X\times_Z Y)$ is order isomorphic to the quotient of $\O X \otimes \O Y$
 by the suplattice congruence generated by $(u \wedge f^*(w)) \otimes v \sim u \otimes (g^*(w) \wedge v)$.
 In terms of this description we have $g^{\prime *}(u) = [u \oplus 1]$ and $f^{\prime *}(v) = [1 \oplus v]$.
 
 Now suppose $g$ is an open map and consider the assignment $[u \otimes v] \mapsto u \wedge f^*g_!(v)$.
 Note that $(u \wedge f^*(w)) \otimes v \mapsto u \wedge f^*(w \wedge g_!(v))$ and $u \otimes (g^*(w) \wedge v) \mapsto u \wedge f^*g_!(g^*(w) \wedge v)$,
 which coincide by openness (and \cref{lem:open_map_characterisation}). Thus, using the suplattice presentation above, this defines a suplattice homomorphism $g'_!\colon \O (X\times_Z Y) \to \O X$.
 
 As the notation suggests, $g'_!$ is left adjoint to $g^{\prime *}$. Indeed, $g'_!g^{\prime *}(u) = u \wedge f^*g_!(1) \le u$ and
 $g^{\prime *}g'_!([u \otimes v]) = [(u \wedge f^*g_!(v)) \otimes 1] = [u \otimes g^*g_!(v)] \ge [u \otimes v]$.
 Finally, we have $g'_!(g^{\prime *}(u') \wedge [u \otimes v]) = g'_!([(u' \wedge u) \otimes v]) = u' \wedge u \wedge f^*g_!(v) = u' \wedge g'_!([u \otimes v])$ and hence $g'$ is open.
\end{proof}

We can now finally prove the characterisation of overt locales that we originally motivated intuitively.

\begin{corollary}
 A locale $X$ is overt if and only if for all locales $Y$ the product projection $\pi_2\colon X \times Y \to Y$ is open.
\end{corollary}
\begin{proof}
 The projection condition gives our definition of overtness if we take $Y = 1$. On the other hand, if $X$ is overt we can consider the following pullback diagram.
 \begin{center}
  \begin{tikzpicture}[node distance=2.5cm, auto]
    \node (A) {$X\times Y$};
    \node (B) [below of=A] {$X$};
    \node (C) [right of=A] {$Y$};
    \node (D) [right of=B] {$1$};
    \draw[->] (A) to node [swap] {$\pi_1$} (B);
    \draw[->] (A) to node {$\pi_2$} (C);
    \draw[->] (B) to node [swap] {${!}$} (D);
    \draw[->] (C) to node {${!}$} (D);
    \begin{scope}[shift=({A})]
        \draw +(0.3,-0.6) -- +(0.6,-0.6) -- +(0.6,-0.3);
    \end{scope}
  \end{tikzpicture}
 \end{center}
 \Cref{thm:open_pullback_stable} gives that $\pi_2$ is open since ${!}\colon X \to 1$ is.
 Explicitly, the left adjoint of $\pi_2^*$ is given by $(\pi_2)_!\colon u \otimes v \mapsto {!}^*\exists(u) \wedge v$.
\end{proof}

It is now easy to see that overt locales are closed under Bishop-finite products.
\begin{lemma}\label{lem:product_of_overt_locales}
 If $X$ and $Y$ are overt, then so is $X \times Y$.
\end{lemma}
\begin{proof}
 The map ${!}\colon X \times Y \to 1$ can be expressed as the composite of $\pi_2\colon X \times Y \to Y$ and ${!}\colon Y \to 1$, which are both open. Thus, it is open too.
\end{proof}
This result can also be understood intuitively: we can existentially quantify over the product by quantifying over each factor in turn.

The following result will help give a Brouwerian counterexample to the claim that all locales are overt.
\begin{lemma}\label{lem:overt_subloc_of_discrete}
 Every overt sublocale of a discrete locale is open.
\end{lemma}
\begin{proof}
 Intuitively, if $V$ is an overt sublocale of the discrete locale $X$, we can verify $x \in V$ by showing $\exists y \in V.\ x = y$.
 More formally, consider the following pullback.
 \begin{center}
  \begin{tikzpicture}[node distance=2.5cm, auto]
   \node (A) {$V$};
   \node (B) [below of=A] {$X$};
   \node (C) [right of=A] {$X \times V$};
   \node (D) [right of=B] {$X \times X$};
   \draw[right hook->] (A) to node [swap] {$i$} (B);
   \draw[->] (A) to node {$(i, \id)$} (C);
   \draw[->] (B) to node [swap] {$(\id, \id)$} (D);
   \draw[right hook->] (C) to node {$X \times i$} (D);
   \begin{scope}[shift=({A})]
     \draw +(0.3,-0.6) -- +(0.6,-0.6) -- +(0.6,-0.3);
   \end{scope}
  \end{tikzpicture}
 \end{center}
 Since the diagonal $(\id, \id)$ is open by \cref{cor:open_diagonal_in_discrete}, so is $(i, \id)$.
 Since $V$ is overt, $\pi_1\colon X \times V \to X$ is open.
 Thus, the composite $i = \pi_1 (i, \id)$ is open.
\end{proof}

\begin{corollary}
 If every closed sublocale of $1$ is overt, then excluded middle holds.
\end{corollary}
\begin{proof}
 By the above this would mean that the closed complement of every truth value is open, and thus every truth value is decidable.
\end{proof}

In fact, overtness is the missing condition we needed to characterise discrete locales.
\begin{proposition}\label{prop:discreteness_characterisation}
 A locale $X$ is discrete if and only if it is overt and has open diagonal.
\end{proposition}
\begin{proof}
 We have already noted that discrete locales are overt and have open diagonal, so we only need to show the converse.
 Let $d$ be the open corresponding to the diagonal inclusion $\Delta\colon X \hookrightarrow X \times X$.
 The idea is to isolate the opens which end up corresponding to singletons.
 We define $S$ to be the set of elements $s \in \O X$ such that $s \oplus s \le d$ and $s > 0$.
 
 We claim that $\O X \cong \Omega^S$. We will show this by giving a bijective suplattice homomorphism $h\colon \Omega^S \to \O X$.
 Since $\Omega^S$ is the free suplattice on $S$, a suplattice map from $\Omega^S$ to $\O X$ corresponds to a function from $S$ to $\O X$. We take this to be the obvious inclusion.
 It remains to show the map $h$ is surjective and injective.
 
 We start by observing that since $\Delta$ is open, $(u \oplus v) \wedge d = (u' \oplus v') \wedge d \iff \Delta^*(u \oplus v) = \Delta^*(u' \oplus v') \iff u \wedge v = u' \wedge v'$.
 Thus, if $u \oplus v \le d$ then $u \oplus v = (u \wedge v) \oplus (u \wedge v)$.
 
 Now since the elements of the form $u \oplus v$ give a base for $X \times X$ we have $d = \bigvee\{u \oplus v \mid u \oplus v \le d\}$.
 By the previous argument we know each $u \oplus v = (u \wedge v) \oplus (u \wedge v)$ and so $d \le \bigvee\{u \oplus u \mid u \oplus u \le d\}$.
 We can project this down to $X$ by applying the left adjoint $(\pi_2)_!$. Note that $\pi_2 \Delta = \pi_2 (\id, \id) = \id$ and so $(\Sloc \pi_2)_!(\Delta\colon X \hookrightarrow X \times X) = X$.
 Thus, we find $1 \le \bigvee\{{!}^*\exists(u) \wedge u \mid u \oplus u \le d\} = \bigvee S$.
 Meeting with $v$ and restricting the join to positive elements (using overtness) we obtain $v \le \bigvee\{s \in S \mid s \le v\}$. Thus, $v = h(S \cap {\downarrow} v)$ and so $h$ is surjective.
 
 We now show $h$ is injective. First note that if $s, s' \in S$ we have $(s \wedge s') \oplus s = (s \wedge s') \oplus (s \wedge s')$ as above.
 So if $s \wedge s' > 0$, projecting onto the second factor we find $s = s \wedge s'$ and so $s \le s'$. Similarly, $s' \le s$ and hence $s = s'$.
 Now suppose $s \le h(T)$. This means $s = \bigvee\{s \wedge t \mid t \in T\} = \bigvee\{s \wedge t \mid t \in T,\, s \wedge t > 0\}$ by overtness. Now $s \wedge t > 0$ means $t = s$ and so $s = \bigvee\{s \mid s \in T\}$ (where $s$ is a fixed element and this denotes a join of a subsingleton, not a join over all of $T$). But this join is inhabited since $s$ is positive and so $s \in T$. It follows that $s \in T \iff s \le h(T)$. Hence, $h(T)$ determines $T$ and $h$ is injective.
\end{proof}

Another important example of an overt locale is given by the reals.
\begin{proposition}
 The locale $\R$ is overt.
\end{proposition}
\begin{proof}
 Observe that the opens $\llround p, q \rrround = \llround p, \infty\rrround \wedge \llround -\infty, q\rrround$ form a base for $\O \R$.
 Furthermore, we have $\llround p, q \rrround = 0$ for $p \ge q$, and so, since the order on $\Q$ is decidable, we can restrict the base to those intervals with $p < q$.
 But it is easy to see that if $p < q$ then $\frac{p+q}{2}$ defines a point of $\R$ that lies in $\llround p, q \rrround$. Hence, $\llround p, q \rrround > 0$ and $\O \R$ has a base of positive elements.
\end{proof}

The following result provides a very general method of showing a locale is overt from a presentation.
\begin{proposition}\label{prop:overtness_from_presentation}
 Let $X$ be a locale with a presentation $\langle G \textup{ $\wedge$-semilattice} \mid R\rangle_\Frm$ of the form used in the coverage theorem (\cref{thm:suplattice_coverage}).
 Then $X$ is overt if and only if there is a set $P \subseteq G$ of candidate positive generators such that
 \begin{enumerate}
  \item $P$ is upward closed,
  \item for each relation $\bigvee S \le \bigvee T$, if $\exists s \in S \cap P$ then $\exists t \in T \cap P$,
  \item for each generator $g \in G$, we have $g \le \bigvee\{ 1 \mid g \in P\}$ in $\O X$.
 \end{enumerate}
\end{proposition}
\begin{proof}
 If $X$ is overt, then the set of positive generators is easily seen to satisfy the hypotheses.
 To prove the converse we will construct a left adjoint to ${!}^*\colon \Omega \to \O X$.
 
 First define a map from $G$ to $\Omega$ by $g \mapsto \llbracket g \in P\rrbracket$. This is monotone by condition (i) and respects the relations by condition (ii). Thus, by the coverage theorem it induces a suplattice homomorphism $\exists\colon \O X \to \Omega$.
 We must show this is indeed a left adjoint to ${!}^*\colon \Omega \to \O X$.
 
 Note that for $g \in G$ we have ${!}^* \exists (g) = \bigvee\{1 \mid g \in P\} \ge g$ by condition (iii) and hence ${!}^*\circ \exists \ge \id_{\O X}$.
 On the other hand, $\exists {!}^*(p) = \bigvee\{ \llbracket 1 \in P \rrbracket \mid p = \top \} = \llbracket 1 \in P \rrbracket \wedge p \le p$ and so $\exists \circ {!}^* \le \id_\Omega$.
 Thus, $\exists$ is left adjoint to ${!}^*$ and $X$ is overt.
\end{proof}

\subsubsection{Compactness}

We now return to the property of compactness. Things proceed similarly to before.

\begin{definition}
 A locale map $f\colon X \to Y$ is \emph{closed} if ${\Sloc f}_!$ maps closed sublocales to closed sublocales.
\end{definition}

\begin{remark}
 Note that a sublocale inclusion is closed if and only if it is a closed sublocale as defined in \cref{def:closed_sublocales}.
\end{remark}

Since a frame map $f^*$ preserves arbitrary joins, it always has a right adjoint $f_*$.

\begin{lemma}\label{lem:closed_map_characterisation}
 A locale map $f\colon X \to Y$ is closed if and only if the frame map $f^*$ and its right adjoint $f_*$ satisfy $f_*({f\,}^*(b) \vee a) = b \vee f_*(a)$.
\end{lemma}
\begin{proof}
 Omitted. (The proof is very similar to that of \cref{lem:open_map_characterisation}.)
\end{proof}

At this point our discussion diverges somewhat from the overt case, since unlike open maps, closed maps are not pullback stable.
Before we deal with this let us define compactness in a more familiar way.
\begin{definition}\label{def:compactness}
 A locale $X$ is \emph{compact} if whenever $\bigvee S = 1$ in $\O X$ then there is some Kuratowski-finite subset $F \subseteq S$ such that $\bigvee F = 1$.
\end{definition}
It will be useful to re-express this condition using the following concept from order theory.
\begin{definition}
 A poset $S$ is called \emph{directed} if every Kuratowski-finite subset $F \subseteq S$ has an upper bound $b \in S$.
\end{definition}
\begin{remark}
 It can be shown by induction that $S$ is directed if and only if it is inhabited and for each pair of elements $s,t \in S$ there is a $b \in S$ with $s,t \le b$.
\end{remark}
\begin{lemma}\label{lem:compactness_via_directed_covers}
 A locale $X$ is compact if and only if whenever $\bigvee S = 1$ is a directed cover, $1 \in S$.
\end{lemma}
\begin{proof}
 If $X$ is compact and $\bigvee S = 1$ then there is a Kuratowski-finite subset $F \subseteq S$ such that $\bigvee F = 1$. But if $S$ is directed, then there is a $b \in S$ such that $b \ge \bigvee F = 1$ and hence $b = 1$.
 
 On the other hand, suppose ever directed cover of $\O X$ contains $1$ and consider an arbitrary cover $S$ of $\O X$. Set $S' = \{\bigvee F \mid F \in \Pfin(S)\}$. This set is directed since $0 = \bigvee \emptyset \in S'$ and if $\bigvee F_1, \bigvee F_2 \in S'$ then $\bigvee F_1 \vee \bigvee F_2 = \bigvee (F_1 \cup F_2) \in S'$, because Kuratowski-finite subsets are closed under binary joins.
 Also, note that $\bigvee S \le \bigvee S'$ and so $S'$ is a cover. Thus, by assumption $1 \in S'$, so $\bigvee F = 1$ for some $F \in \Pfin(S)$. Consequently, $X$ is compact.
\end{proof}

\begin{corollary}\label{cor:compactness_right_adjoint}
 A locale $X$ is compact if and only if the right adjoint ${!}_*\colon \O X \to \Omega$ preserves directed suprema.
\end{corollary}
\begin{proof}
 By the usual argument that maps into $\Omega$ are determined by what they send to $\top$ it suffices to show that ${!}_*$ preserves directed covers.
 Note that $1 = {!}^*(\top) \le u \iff \top \le {!}_*(u)$ and so ${!}_*(u) = \top \iff u = 1$.
 Thus, ${!}^*$ preserves directed joins if and only if for $S$ directed, $\bigvee S = 1$ implies $\exists s \in S.\ s = 1$, which is simply the characterisation of compactness from \cref{lem:compactness_via_directed_covers}.
\end{proof}

This motivates the following strengthening of closedness, which will behave more similarly to openness.
\begin{definition}
 A locale map $f\colon X \to Y$ is \emph{proper} if it is closed and $f_*$ preserves directed suprema.
\end{definition}

\begin{proposition}
 A locale $X$ is compact if and only if the unique map ${!}\colon X \to 1$ is proper.
\end{proposition}
\begin{proof}
 By \cref{cor:compactness_right_adjoint} it simply remains to show that if $X$ is compact, then ${!}\colon X \to 1$ is closed.
 
 We must show ${!}_*({!}^*(p) \vee a) \le p \vee {!}_*(a)$. (The reverse inequality always holds.)
 Suppose ${!}_*({!}^*(p) \vee a) = \top$. Then ${!}^*(p) \vee a = 1$. Observe that ${!}^*(p) \vee a = \bigvee\{1 \mid p = \top\} \vee a = \bigvee(\{a\} \cup \{1 \mid p\})$. This join is directed and so by compactness, either $a = 1$ or $p = \top$. But this means $p \vee {!}_*(a) = \top$, as required.
\end{proof}

\begin{remark}
 Inclusions of closed sublocales are another example of proper maps. In fact, in that case the right adjoints (which are of the form $u \mapsto u \vee a$) preserve all inhabited suprema.
\end{remark}

Just like open maps, proper maps are pullback stable. In fact, we have the following characterisation, the proof of which is unfortunately beyond the scope of these notes.
See \citet*{vermeulen1994proper} for details.
\begin{theorem}\label{thm:pullback_stable_closed_maps}
 Proper maps are precisely the pullback-stable closed maps.
\end{theorem}
\begin{remark}
 The proof that proper maps are pullback stable is similar to that for open maps, but makes use of `preframes' instead of suplattices. A \emph{preframe} is a poset with finite meets and directed joins in which finite meets distribute over directed joins. Schematically, we have frames = suplattices + finite meets = preframes + finite joins.
\end{remark}

We can now give the characterisation of compactness we motivated in terms of universal quantification.
\begin{corollary}
 A locale $X$ is compact if and only if for all locales $Y$ the product projection $\pi_2\colon X \times Y \to Y$ is closed.
\end{corollary}
\begin{proof}
 This is simply \cref{thm:pullback_stable_closed_maps} applied to ${!}\colon X \to 1$.
\end{proof}

Analogous preservation results to those we proved for overtness can now be given for compactness. They are proved in essentially the same way (and have similar intuitive explanations from the perspective of verifiability).

\begin{lemma}\label{lem:closed_subloc_of_compact}
 Closed sublocales of compact locales are compact.
\end{lemma}
\begin{lemma}
 Images of compact (sub)locales are compact.
\end{lemma}
\begin{lemma}
 Bishop-finite products of compact locales are compact.
\end{lemma}

We also have an analogue of \cref{lem:overt_subloc_of_discrete} (the proof of which did not actually use the overtness of the parent locale).
\begin{lemma}
 Every compact sublocale of a Hausdorff locale is closed.
\end{lemma}

Let us consider some examples of compact locales.
\begin{proposition}
 A set $X$ is compact as a locale if and only if it is Kuratowski-finite.
\end{proposition}
\begin{proof}
 Suppose $X$ is compact. We have $X = \bigcup_{x \in X} \{x\}$ and so by compactness, $X$ is a Kuratowski-finite join of singletons. Hence, $X$ is Kuratowski-finite.
 
 Now suppose $X$ is Kuratowski-finite. Then $X$ is the image of some set of the form $[n] = \{m \in \N \mid m < n\}$. Thus, it suffices to show $[n]$ is compact.
 
 We proceed by induction.
 Certainly, $[0]$ is compact.
 Suppose $[n]$ is compact and consider a union $\bigcup \mathscr{S} = [n] \cup \{n\}$.
 Then $[n] \subseteq \bigcup \mathscr{S}$ and so there is a Kuratowski-finite subset $\mathscr{F} \subseteq \mathscr{S}$ such that $[n] \subseteq \mathscr{F}$.
 Moreover, $n \in S$ for some $S \in \mathscr{S}$. Thus, $\mathscr{F} \cup \{S\} \subseteq \mathscr{S}$ is a Kuratowski-finite subcover. So $[n+1]$ is compact.
\end{proof}
\begin{corollary}
 A set $X$ is compact Hausdorff if and only if it is Bishop-finite.
\end{corollary}
\begin{proof}
 Recall that a set is Bishop-finite if and only if it is Kuratowski-finite and has decidable equality.
\end{proof}

Here we see that our different notions of finiteness again have familiar topological generalisations.
\begin{remark}
 From this perspective it is perhaps less surprising that a subset of a Kuratowski-finite set can fail to be Kuratowski-finite. After all, we are familiar with the idea that an open sublocale of a compact locale need not be compact.
\end{remark}

We can now discuss a slightly more complicated example. Recall that the Cantor space $2^\N$ has the presentation \[\O(2^\N) \cong \langle z_n, u_n,\ n \in \N \mid z_n \wedge u_n = 0, z_n \vee u_n = 1\rangle.\]
If $\bigvee S = 1$ in $\O(2^\N)$ this should be forced by some relations and since every relation involves only finite joins, we might intuitively feel that we should be able to find a finite subset of $S$ which already joins to give $1$. This intuition can be made precise.
\begin{proposition}
 If every join in a frame presentation is finite, then the resulting locale is compact.
\end{proposition}
\begin{proof}
 Since the presentation involves only finite joins and finite meets, it can also be viewed as a presentation for a distributive lattice $L$. The presented frame is then the free frame on this distributive lattice $L$. (This frame is also that obtained in \cref{ex:stone_spectrum}.)
 It is not hard to show that this free frame is given explicitly by the frame $\Idl L$ of lattice ideals on $L$ --- that is, downsets on $L$ which are closed under finite joins.
 Directed joins in $\Idl L$ are simply given by union and hence if $\bigvee_\alpha J_\alpha = {\downarrow} 1$ is a directed join then we must have $1 \in J_\alpha$ for some $\alpha$. Thus, $\Idl L$ is compact.
\end{proof}
\begin{remark}
 In fact, analogously to \cref{prop:overtness_from_presentation} there is a more general result for showing a locale is compact from a presentation. See Theorem 10 of \citet*{vickers2006compactness}.
\end{remark}

We can now use the compactness of $2^\N$ to deduce that the closed real interval $[0,1]$ is also compact by constructing an epimorphism $f\colon 2^\N \to [0,1]$ which, intuitively, interprets the points of $2^\N$ as the binary expansions of real numbers. We omit the details.

\begin{remark}\label{rem:spatial_interval_noncompact}
 You might have heard constructivists claim that Cantor space or the closed real interval might fail to be compact constructively. How can this be reconciled with the above results?
 The crucial point is that those constructivists are not referring to the true localic versions of $2^\N$ or $[0,1]$, but instead to their spaces of points. The failure of these spaces to be compact is simply the failure of these locales to be spatial. This is an example of the localic approach being much better behaved than the point-set approach in the constructive setting. One topos where $2^\N$ and $[0,1]$ fail to be spatial is the effective topos. See \citet*{bauer2006konig} for details.
\end{remark}

\section{Applications}

Now that we have seen the basic ideas behind constructive pointfree topology, we will take a look at some of the advantages of adopting a more topological approach to constructive mathematics. We will also see an example of how constructive topology can be useful even if you only care about proving classical theorems.

\subsection{Ostensibly classical theorems that always hold for locales}

A number of topological results are famously nonconstructive and one might fear that this might make it difficult to do topology or analysis without them.
However, we typically find that some formulation of the results do in fact hold for locales.
Indeed, pointfree results tend to be remarkably robust with respect to precisely which foundational axioms are assumed.

\subsubsection{Tychonoff's theorem}

Perhaps the most famous example of a nonconstructive theorem in topology is Tychonoff's theorem that an arbitrary product of compact spaces is compact.
This is known to be equivalent to the axiom of choice. However, its localic variant is constructively valid. (The discrepancy is resolved by observing that the product might fail to be spatial even if all the factors are.)

\begin{theorem}
 Arbitrary products of compact locales are compact.
\end{theorem}

The proof is somewhat too involved to give here, but see \citet*{vickers2005tychonoff} for details. A simple classical proof is also given in \citet*{johnstone1991preframe}, which can be made to work constructively if we restrict to indexing sets with decidable equality.

\subsubsection{The axiom of choice}

It is not only explicitly topological results that become constructively valid with the pointfree approach. Even classical set-theoretic theorems can often become constructive if replace the sets with locales. The axiom of choice itself provides such an example.

Recall that the axiom of choice can be formulated as saying that a product of inhabited sets is inhabited.
Constructively, we still have the following result from \citet*{henry2016corrected}.
\begin{theorem}
 A product of positive overt locales indexed by a set with decidable equality is positive and overt.
\end{theorem}
\begin{proof}
 Let $(X_\alpha)_{\alpha \in I}$ be a family of positive overt locales.
 Recall that $\O (\prod_{\alpha \in I} X_\alpha)$ has a presentation with a generator $\iota_\alpha(u)$ for each $\alpha \in I$ and $u \in \O X_\alpha$ and relations enforcing the relations holding in each $\O X_\alpha$. We can obtain a set of basic opens for the product by taking (Kuratowski-)finite meets of these generators. Explicitly, the basic opens are of the form $\bigwedge_{\alpha \in F} \iota_\alpha(u_\alpha)$ for $F \in \Pfin(I)$.
 
 If the indexing set $I$ has decidable equality, then so does each Kuratowski-finite subset $F \subseteq I$ and hence such subsets are \emph{Bishop}-finite.
 In other words, each such finite meet involves opens from \emph{distinct} factors.
 
 We now claim that $\bigwedge_{\alpha \in F} \iota_\alpha(u_\alpha)$ is positive if and only if $u_\alpha$ is positive for each $\alpha \in F$.
 (Note that it is important that each $u_\alpha$ comes from a different $X_\alpha$, since if $a$ and $b$ lie in the same factor, their intersection might fail to be positive even if they both are.)
 We will show that the map $\bigwedge_{\alpha \in F} \iota_\alpha(u_\alpha) \mapsto \bigwedge_{\alpha \in F} \llbracket u_\alpha > 0\rrbracket$ defines a suplattice homomorphism that is left adjoint to ${!}^*\colon \Omega \to \O (\prod_{\alpha \in I} X_\alpha)$ as in \cref{prop:overtness_from_presentation}.
 
 We would like to apply the coverage theorem and so we must find a presentation for $\O (\prod_{\alpha \in I} X_\alpha)$ that is of the appropriate form.
 The $\wedge$-semilattice of generators will consist of formal finite meets $\bigwedge_{\alpha \in F} \iota_\alpha(u_\alpha)$ for $F \in \Pfin(I)$. The meet operation is \[\bigwedge_{\alpha \in F_1} \iota_\alpha(u_\alpha) \wedge \bigwedge_{\alpha \in F_2} \iota_\alpha(v_\alpha) = \bigwedge_{\alpha \in F_1 \cup F_2} \iota_\alpha( \bigwedge (\{ u_\alpha \mid \alpha \in F_1 \} \cup \{ v_\alpha \mid \alpha \in F_2 \})).\]
 We should take a quotient so that $\iota_\beta(1) = 1$ for each $\beta \in I$. That is, we have $\iota_\beta(1) \wedge \bigwedge_{\alpha \in F} \iota_\alpha(u_\alpha) = \bigwedge_{\alpha \in F} \iota_\alpha(u_\alpha)$ for all $F \in \Pfin(I)$.
 The relations in the frame presentation then consist of inequalities \[\bigvee_{u \in U} \bigwedge_{\alpha \in F} \iota_\alpha(a_\alpha) \wedge \iota_\beta(u) \le \bigvee_{v \in V} \bigwedge_{\alpha \in F} \iota_\alpha(a_\alpha) \wedge \iota_\beta(v)\] whenever $\bigvee U \le \bigvee V$ for $U,V \subseteq \O X_\beta$.
 It is not too difficult to show that this is indeed a presentation for the product locale.
 
 Note that for the assignment $\bigwedge_{\alpha \in F} \iota_\alpha(u_\alpha) \mapsto \bigwedge_{\alpha \in F} \exists_{X_\alpha}(u_\alpha)$ to be well-defined map on the generators, we require that $\iota_\beta(1) = 1$ for all $\beta \in I$, but this is true since each $X_\beta$ is positive by assumption. Also note that this map is monotone. Now by the coverage theorem, this induces a suplattice homomorphism $\exists\colon \O (\prod_{\alpha \in I} X_\alpha) \to \Omega$ as long as it satisfies the necessary relations, which can easily be seen to hold.
 
 Finally, we show that $\exists$ is left adjoint to ${!}^*\colon \Omega \to \O (\prod_{\alpha \in I} X_\alpha)$.
 We have
 \begin{align*}
  {!}^*\exists (\bigwedge_{\alpha \in F} \iota_\alpha(u_\alpha)) &= \bigwedge_{\alpha \in F} {!}^* \exists_{X_\alpha}(u_\alpha) \\
                                                                 &= \bigwedge_{\alpha \in F} \iota_\alpha {!}_{X_\alpha}^* \exists_{X_\alpha}(u_\alpha) \\
                                                                 &\ge \bigwedge_{\alpha \in F} \iota_\alpha(u_\alpha)
 \end{align*}
 and so ${!}^* \circ \exists \ge \id$. On the other hand, $\exists {!}^*(p) = \bigvee \{ \exists(1) \mid p \} = {\bigvee \{\top \mid p\}} = p$ and so $\exists \circ {!}^* = \id$. Thus, $\exists$ is a left adjoint retraction of ${!}^*$ and hence $\prod_{\alpha \in I} X_\alpha$ is overt and positive, as required.
\end{proof}

Recall that sets are always overt, and a set is positive if and only if is inhabited. Furthermore, under the assumption of excluded middle, every set has decidable equality.
So in this case a product of inhabited sets in $\Loc$ is always a positive locale. The set-theoretic axiom of choice is then equivalent to the claim that every such product locale is not just positive, but has a \emph{point}.

\subsubsection{Excluded middle}

Finally, there is a sense in which excluded middle itself holds from the localic perspective.
Of course, a truth value $p \in \Omega$ is not guaranteed to have a true lattice-theoretic complement and so the closest approximation to such a complement $\neg p$ does not need to satisfy $p \vee \neg p = \top$ or $\neg\neg p = p$.
On the other hand, we have already seen that the open sublocale $P$ of $1$ corresponding to the open $p \in \Omega = \O(1)$ always has a closed complement $P\comp \hookrightarrow 1$ in the lattice of all sublocales of $1$.

It is not that the closed sublocale $P\comp$ fails to exist constructively, but that $P\comp$ is not necessarily an open sublocale (i.e.\ a subset) of $1$. Indeed, the locale $P\comp$ cannot be shown to be a discrete locale at all. The open sublocale given by the intuitionistic negation of $p$ is the \emph{interior} of the closed sublocale $P\comp$ and the `bad behaviour' of constructive negation is simply a reflection of the fact that we are forcing these locales to be sets.

A slightly different perspective on the same phenomenon is given in \citet*{vickers2022}. In $\Set$, exponentials involving subsingletons are given by implication of truth values. In particular, $\emptyset^{\{ {*} \mid p = \top \}}$ is isomorphic to $\{ {*} \mid \neg p = \top \}$. In $\Loc$ on the other hand, the exponential object $0^P = \prod_p 0$ is isomorphic to the compact Hausdorff locale $P\comp$. It can then be shown that $0^{0^P} \cong P$. This is another sense in which double negation elimination holds constructively for locales.

\subsection{An extended example: the extreme value theorem}

Let us conclude this chapter with an example of a result from analysis proved using the constructive pointfree approach. We will then see how it might be applied to deduce a classical theorem in fibrewise topology. 

We will prove a constructive version of the \emph{extreme value theorem}. Classically, this states that if $X$ is a compact topological space, a continuous function $f\colon X \to \R$ attains a maximum value.
Our approach is based on that given in \citet*{taylor2010analysis}.

\subsubsection{One-sided real numbers}

We will need to understand suprema of real numbers. Classically, all inhabited, bounded sets of reals have suprema and infima, but constructively this is not the case.
Suppose the set $\{0\} \cup \{1 \mid p = \top\}$ had a supremum $s$. By the locatedness property of Dedekind reals (see the start of \cref{sec:classifying_locales}) we have $s > 0 \lor s < 1$, but in the first case $p = \top$ and in the second $p \ne \top$, and hence excluded middle follows.

To discuss suprema and infima we must use slightly different variants of the reals that have a distinct constructive nature. Completing $\Q$ under inhabited suprema will give the \emph{lower reals}.
These are given by \emph{one-sided} Dedekind cuts consisting only of the sets $L$ of lower bounds instead of pairs $(L,U)$.
Similarly, completing $\Q$ under inhabited infima gives the \emph{upper reals}, which are constructed from the sets of upper bounds $U$.
Thus, lower reals are approximated from below and upper reals are approximated from above.

The formal axiomatisation of the lower reals $\lowerReals$ involves a basic proposition $\ell_q$ for each $q \in \Q$ and the three families of axioms of Dedekind cuts that refer only to $\ell_q$ ---
namely, the following.
\begin{displaymath}
\begin{array}{l@{\qquad\quad}r@{\hspace{1.5ex}}c@{\hspace{1.5ex}}l@{\quad}@{}l@{\qquad\quad}r@{}}
 (1) & \ell_q &\vdash& \ell_p & \text{ for $p \le q$} & \text{($L$ downward closed)} \\
 (2) & \ell_p &\vdash& \bigvee_{q > p} \ell_q & \text{ for $p \in \Q$} & \text{($L$ rounded)} \\
 (3) & \top &\vdash& \bigvee_{q \in \Q} \ell_q && \text{($L$ inhabited)}
\end{array}
\end{displaymath}
This gives the presentation
\begin{align*}
 \O \lowerReals = \langle \ell_q,\ q \in \Q \mid {} & \ell_p = \bigvee_{q > p} \ell_q, \bigvee_{q \in \Q} \ell_q = 1 \rangle.
\end{align*}
Similarly, for the upper reals $\upperReals$ we have
\begin{align*}
 \O \upperReals = \langle u_q,\ q \in \Q \mid {} & u_q = \bigvee_{p < q} u_p, \bigvee_{q \in \Q} u_q = 1 \rangle.
\end{align*}
Classically, $\lowerReals$ is the space with underlying set $\R \sqcup \{\infty\}$ and the topology of lower semicontinuity. This has a base of opens of the form $\{(q, \infty]\}$ for $q \in \Q$.
Similarly, $\upperReals$ is $\R \sqcup \{-\infty\}$ with the topology of upper semicontinuity.

Of course, the locales $\lowerReals$ and $\upperReals$ are actually isomorphic (by negation), but it will be convenient to distinguish them.
The familiar locale $\R$ of reals (also called \emph{two-sided} or \emph{Dedekind} reals in this context) is a sublocale of $\lowerReals \times \upperReals$.

Since it is cumbersome to work with sets of rationals explicitly, in what follows we will simply use the letter $r$ to denote a lower, upper or two-sided real and implicitly define the sets $L$ or $U$ (as appropriate) by saying when $q < r$ or $r < q$ for rational $q$.

\subsubsection{Suprema over sublocales}

Note that classically the supremum $r$ of a set $X \subseteq \R$ is characterised by $r \le u$ if and only if $u \in \R$ is an upper bound of $X$ (i.e.\ $\forall x \in X.\ x \le u$).
To define the supremum as a lower real we must specify for which $q \in \Q$ we have $q < \sup X$. We then have $q < \sup X \iff \neg(\sup X \le q) \iff \neg(\forall x \in X.\ x \le q) \iff \exists x \in X.\ q < x$.

Since we can existentially quantify over overt locales this suggests we can define the supremum of an overt sublocale $V$ of $\R$ as a lower real $r$ by
\[q < r \iff \exists x \in V.\ q < x.\]
The existential quantifier is interpreted by restricting the open $\llround q, \infty\rrround$ to $V$ (using the underlying frame homomorphism of the inclusion) and then applying $\exists\colon \O V \to \Omega$. If this gives $\top$ we write $V \between \llround q, \infty\rrround$ which is understood to mean that $V$ meets (i.e.\ intersects) $\llround q, \infty\rrround$.
\begin{lemma}
 If $V$ is a positive overt sublocale of $\R$, then
 $q < r \iff V \between \llround q, \infty\rrround$ indeed defines a lower real $r$.
\end{lemma}
\begin{proof}
 In $\O\R$ we have $\llround p, \infty\rrround = \bigvee_{q > p} \llround q, \infty\rrround$. Applying the join-preserving map $V \between (-)$ we obtain $V \between \llround p, \infty\rrround \iff \exists q > p.\ V \between \llround q, \infty\rrround$ and so we find $p < r \iff \exists q > p.\ q < r$. Thus, axioms (1) and (2) hold.
 
 Similarly, $\bigvee_{q \in \Q} \llround q, \infty\rrround = 1$ in $\O\R$ gives $\exists q \in \Q.\ q < r \iff V \between 1$. But $V \between 1$ is precisely the claim that $V$ is positive and so axiom (3) follows and $r$ is a lower real.
\end{proof}
\begin{remark}
 Examining the proof it is apparent that we only used the axioms of $\R$ involving the generators of the form $\ell_q$ and so a similar result holds for positive overt sublocales of $\lowerReals$.
\end{remark}

Suprema of compact locales can also be defined, but this time we obtain an \emph{upper} real. We define the supremum $r$ of a compact sublocale $K$ of $\R$ by
\[r < q \iff \forall x \in K.\ x < q.\]
Note that this is similar to the definition in our informal discussion above except that we restrict to rational $q$ and replace $\le$ with strict inequality.

\begin{lemma}
 If $K$ is a compact sublocale of $\R$, then
 $r < q \iff K \le \llround -\infty, q\rrround$ indeed defines an upper real $r$.
\end{lemma}
\begin{proof}
 Similarly to before we have $\llround -\infty, q\rrround = \bigvee_{p < q} \llround -\infty, p\rrround$ in $\O\R$.
 Thus, $K \le \llround -\infty, q\rrround \iff K \le \bigvee_{p < q} \llround -\infty, p\rrround$.
 Now note that the join is directed and so $K \le \bigvee_{p < q} \llround -\infty, p\rrround$ if and only if $K \le \llround -\infty, p\rrround$ for some $p < q$.
 Thus, we can conclude $r < q \iff \exists p < q.\ r < p$ and hence the upward closure and roundedness axioms for upper reals are satisfied.
 
 Now consider $\bigvee_{q \in \Q} \llround -\infty, q \rrround = 1$ in $\O\R$. We always have $K \le 1$ and so by compactness $K \le \llround -\infty, q \rrround$ for some $q \in \Q$. Thus, $r < q$ for some $q \in \Q$ and the inhabitedness axiom holds.
\end{proof}

\subsubsection{The extreme value theorem}

If $K$ is a sublocale of $\R$ that is positive, overt \emph{and} compact, then the above definitions define the supremum both as a lower real and as an upper real.
As we might expect, these give the two sides of a single Dedekind cut.

To prove this makes the supremum $r$ a Dedekind real it remains to prove the disjointness and locatedness axioms.
Disjointness means that if $p < r$ and $r < q$ then $p < q$. By the definition of $r$, the assumption $p < r$ means $K \between \llround p, \infty\rrround$ and $r < q$ means $K \le \llround -\infty, q\rrround$. Thus, $\llround -\infty, q\rrround \between \llround p, \infty\rrround$, or equivalently $\llround -\infty, q\rrround \wedge \llround p, \infty\rrround > 0$. But by the disjointness axiom of $\O \R$ this implies $p < q$, as required.

Locatedness means that if $p < q$ then either $p < r$ or $r < q$. Suppose $p < q$. We must show $K \between \llround p, \infty\rrround$ or $K \le \llround -\infty, q\rrround$.
By the locatedness axiom in $\O \R$ we know that $\llround p, \infty\rrround \vee \llround -\infty, q\rrround = 1$. Thus, the desired result follows from the following lemma.

\begin{lemma}
 Let $K$ be a compact overt locale and suppose $u \vee v = 1$ in $\O K$.
 Then either $u > 0$ or $v = 1$.
\end{lemma}
\begin{proof}
 By overtness we have $u = \bigvee\{u \mid u > 0\}$. Then by compactness $1 = u \vee v = \bigvee\{u \mid u > 0\} \vee v$ has a Kuratowski-finite subcover.
 So we have $c_1 \vee c_2 \vee \dots \vee c_n = 1$ where for each $i$, either $c_i = v$ or $c_i = u > 0$.
 Since there are finitely many of these, we can conclude that either $c_i = v$ for all $i$ or $u > 0$.
 In the former case $v = 1$ and in the latter case $u > 0$, as required.
\end{proof}

Finally, we note that $r$ lies in the sublocale $K$ and is thus a \emph{maximum}.
This is because by definition whether $r$ lies in a subbasic open is defined purely in terms of that open's restriction to $K$
and hence the map $r^*\colon \O \R \to \Omega$ factors through the frame quotient $\O \R \twoheadrightarrow \O K$.

Therefore, up to taking images, we have proved the following.

\begin{theorem}[Extreme value theorem]\label{thm:extreme_value_theorem}
 If $X$ is any positive compact overt locale and $f\colon X \to \R$, then $f$ has a supremum given by a Dedekind real number $r$ defined by
 $q < r \iff f^* \llround q, \infty\rrround > 0$ and $r < q \iff f^*\llround -\infty, q\rrround = 1$.
 Furthermore, this supremum is a maximum in the sense that ${\Sloc f}^*(\{r\})$ is (compact and) positive.
\end{theorem}
\begin{proof}
 Simply take $K = {\Sloc f}_!(X)$. The image of a positive compact overt sublocale is positive, compact and overt. The definition of $r$ follows since ${\Sloc f}_!(X) \le u \iff X \le f^*(u)$ by adjointness and ${\Sloc f}_!(X) \between u \iff X \between f^*(u)$ by a simple extension of \cref{cor:image_and_preimage_of_positive_locales}.
 
 Finally, we show ${\Sloc f}^*(\{r\})$ is positive. Since $\R$, and thus $K$, is Hausdorff, the inclusion $\{r\} \hookrightarrow K$ is closed.
 Now note that the map $f$ is epic onto its image and consider the following pullback diagram.
 \begin{center}
 \begin{tikzpicture}[node distance=2.5cm, auto]
  \node (A) {${\Sloc f}^*(\{r\})$};
  \node (B) [below of=A] {$X$};
  \node (C) [right of=A] {$\{r\}$};
  \node (D) [right of=B] {$K$};
  \node (E) [right of=C] {$\{r\}$};
  \node (F) [right of=D] {$\R$};
  \draw[right hook->] (A) to node [swap] {} (B);
  \draw[->>] (A) to node {} (C);
  \draw[->>] (B) to node [swap] {} (D);
  \draw[->, bend right=20] (B) to node [swap] {$f$} (F);
  \draw[right hook->] (C) to node {} (D);
  \draw[double equal sign distance] (C) to node {} (E);
  \draw[right hook->] (E) to node {} (F);
  \draw[right hook->] (D) to node {} (F);
  \begin{scope}[shift=({A})]
   \draw +(0.3,-0.6) -- +(0.6,-0.6) -- +(0.6,-0.3);
  \end{scope}
  \begin{scope}[shift=({C})]
   \draw +(0.3,-0.6) -- +(0.6,-0.6) -- +(0.6,-0.3);
  \end{scope}
 \end{tikzpicture}
 \end{center}
 It is not hard to show that epimorphisms of locales are stable under pullback along closed inclusions, and hence the map ${\Sloc f}^*(\{r\}) \to \{r\}$ is epic.
 It follows that ${\Sloc f}^*(\{r\})$ is positive.
 (It is compact, because it is a closed sublocale of $X$.)
\end{proof}

\begin{remark}\label{rem:no_pointset_EVT}
 As we might have come to expect by this point, the extreme value theorem is \emph{not} a theorem of constructive \emph{point-set} topology. In particular, the positive locale ${\Sloc f}^*(\{r\})$ might fail to contain a point.
 There are still some constructive point-set results about the existence of suprema and infima of certain sets of reals (see \citet*[Section 6.1]{troelstra1988book}), but these can be derived as consequences of our pointfree theorem.
\end{remark}

\begin{remark}
 It is also interesting to consider how the supremum of a positive compact overt sublocale of $\R$ changes as the sublocale is varied.
 We can understand this by using a localic version of the Vietoris hyperspace construction (see \citet*{johnstone1985vietoris}). The points of the (positive) Vietoris hyperlocale $V_{+}(X)$ are the positive compact overt sublocales of $X$ (see \citet*{vickers1997powerlocale}), and the supremum becomes a locale morphism from $V_{+}(\R)$ to $\R$ (by essentially the same argument we gave above).
\end{remark}

\subsubsection{An application to fibrewise topology}\label{sec:fibrewise_application}

Now we can discuss an example of the pay-off we get from proving \cref{thm:extreme_value_theorem} constructively.
As we mentioned in \cref{sec:why_work_constructively}, constructive results can be interpreted in any topos.
In particular if $B$ is a locale, we can interpret the extreme value theorem in the topos $\Sh(B)$ of sheaves over $B$.
Let us see what this gives us.

An in-depth discussion of how to interpret intuitionistic logic in $\Sh(B)$ is beyond the scope of these notes. See \citet*[Part C]{elephant2} for details.
Instead, we provide the following dictionary to translate between internal constructive concepts and their externalisations.
In particular, we will use that there is an equivalence of categories between locales internal to $\Sh(B)$ and the slice category $\Loc/B$.

\begin{table}[H] %
\begin{tabular}{>{\raggedright\arraybackslash}m{7.5cm}>{\raggedright\arraybackslash}m{6.5cm}}
  \toprule
  \textbf{Constructive concept} & \textbf{Interpretation in $\Sh(B)$} \\
  \midrule
  A locale $X$ & A locale map $\chi\colon X \to B$  \\
  \rowcolor{gray!20}
  The locale $\R$ & The projection $\pi_1\colon B \times \R \to B$ \\
  A locale morphism & A morphism in $\Loc/B$ \\
  \rowcolor{gray!20}
  A point of $X$ & A section of $\chi$ \\
  $X$ is overt & $\chi$ is open \\
  \rowcolor{gray!20}
  The map $\exists\colon \O X \to \Omega$ for $X$ overt & The left adjoint $\chi_!\colon \O X \to \O B$ \\
  $X$ is overt and positive & $\chi$ is open and epic \\
  \rowcolor{gray!20}
  $X$ is compact & $\chi$ is proper \\
  The map ${!}_*\colon \O X \to \Omega$ for $X$ compact & The right adjoint $\chi_*\colon \O X \to \O B$ \\
  \rowcolor{gray!20}
  $X$ is compact and positive & $\chi$ is proper and epic \\
  \bottomrule
\end{tabular}
\end{table}

Using this we can now see that interpreting \cref{thm:extreme_value_theorem} in $\Sh(B)$ immediately gives the following result.
\begin{theorem}\label{cor:fibrewise_localic_extreme_value}
Let $\chi\colon X \to B$ be a locale map that is epic, proper and open and let $f\colon X \to \R$ be any locale map.
Then there is a locale map $\rho\colon B \to \R$ defined by
\begin{align*}
 \rho^*\llround q, \infty\rrround &= \chi_!f^* \llround q, \infty\rrround \\
 \rho^*\llround -\infty, q\rrround &= \chi_*f^*\llround -\infty, q\rrround
\end{align*}
taking the fibrewise supremum\footnote{The above specification of $\rho$ is the \emph{definition} of fibrewise supremum by the same argument we gave when defining suprema of real numbers above. See also the proof of \cref{cor:fibrewise_topological_extreme_value} below.} of $f$.
\begin{center}
  \begin{tikzpicture}[node distance=2.5cm, auto]
    \node (X) {$X$};
    \node (R) [right of=X] {$B \times \R$};
    \node (B) [below of=R] {$B$};
    \draw[->>] (X) to node [swap] {$\chi$} (B);
    \draw[->] (X) to node {$(\chi, f)$} (R);
    \draw[->] (R) to node [swap] {$\pi_1$} (B);
    \draw[->,dashed,bend right] (B) to node [swap] {$(\id_B,\rho)$} (R);
  \end{tikzpicture}
\end{center}
Moreover, this is a fibrewise \emph{maximum} in the sense the canonical morphism into $B$ from the sublocale $M$ of $X$ where the maximum is attained is a (proper) epimorphism.
\begin{center}
 \begin{tikzpicture}[node distance=2.5cm, auto]
  \node (A) {$M$};
  \node (B) [below of=A] {$X$};
  \node (C) [right of=A] {$B$};
  \node (D) [right of=B] {$B \times \R$};
  \draw[right hook->] (A) to node [swap] {} (B);
  \draw[->>] (A) to node {} (C);
  \draw[->] (B) to node [swap] {$(\chi,f)$} (D);
  \draw[right hook->] (C) to node {$(\id_B,\rho)$} (D);
  \begin{scope}[shift=({A})]
   \draw +(0.3,-0.6) -- +(0.6,-0.6) -- +(0.6,-0.3);
  \end{scope}
 \end{tikzpicture}
\end{center}
\end{theorem}
So we have obtained a `fibrewise' extreme value theorem for free!

\begin{remark}
 Suppose that contrary to \cref{rem:no_pointset_EVT} the constructive extreme value theorem proved the existence of a \emph{point} achieving the maximum.
 Then in the context of \cref{cor:fibrewise_localic_extreme_value} we would find that the map $\chi$ always has a family of continuous local sections which pick out points that attain the maximum in each fibre.
 But note that if $X = B = \C$, $\chi(x) = x^2$ and $f(x) = \left\vert x \right\vert$, then no such section exists in the neighbourhood of $0$, since there is no continuous choice of square roots.
\end{remark}

While it is nice to obtain results for free, a skeptic might point out that \cref{cor:fibrewise_localic_extreme_value} is still a result of constructive pointfree topology and so what use is it to classical mathematicians? Of course, every constructive result is also a classical result. Moreover, results in point-set topology often follow from their pointfree counterparts. In this case, we have the following corollary (assuming classical logic).

\begin{corollary}\label{cor:fibrewise_topological_extreme_value}
 Let $\chi\colon X \to B$ a proper and open quotient of topological spaces and let $f\colon X \to \R$ be a continuous function.
 Then the function $\rho\colon B \to \R$ defined by $\rho(b) = \max_{x \in \chi^{-1}(b)} f(x)$ is continuous.
\end{corollary}
\begin{proof}
 It can be shown that if a continuous function between topological spaces is proper or open then the corresponding locale map is too.
 Thus, \cref{cor:fibrewise_localic_extreme_value} applies and we obtain a map from the locale associated to $B$ to $\R$, which in turn induces a continuous map $s\colon B \to \R$.
 
 This map $s$ is equal to $\rho$ by construction, but let us spell out the formal details.
 In spatial terms the map $\chi_!$ is simply the image of $\chi$ restricted to open sets.
 Hence, we have $s^{-1}( (q, \infty) ) = \chi(f^{-1} (q,\infty) )$ and so $s(b) > q \iff b \in \chi(f^{-1} (q,\infty) ) \iff \exists x \in \chi^{-1}(b).\ f(x) > q \iff \sup_{x \in \chi^{-1}(b)} f(x) > q$.
 Thus, $s(b) = \sup_{x \in \chi^{-1}(b)} f(x)$, as required. (Note that this supremum is indeed a maximum by the extreme value theorem and because $\chi^{-1}(\{b\})$ is compact since $\chi$ is proper.)
\end{proof}

\begin{remark}
 This result can be understood as a variant of Berge's maximum theorem (see \citet*[Section 17.5]{charalambos2006analysis}) where the compact-valued correspondence is given by the preimage of a proper and open continuous map.
\end{remark}

Therefore, the constructive approach has yielded a result that should be interesting even to an avowedly classical mathematician. In contrast, we would not have been able to obtain \cref{cor:fibrewise_topological_extreme_value} from the classical extreme value theorem without additional work.

\section{Further reading}

In these notes I have tried to give a taste of constructive pointfree topology, as well as some idea of its advantages over both classical topology and non-topological approaches to constructive mathematics. There is, of course, much that I have not been able to cover.
If you are interested to learn more, I encourage you to look at the following resources.

\begin{itemize}

\item \citet*{picado2012book} is a nice introduction to the classical theory of pointfree topology. It is fairly comprehensive
and can be a good reference.
Of the resources I recommend here, this demands the least of the reader. Unfortunately, the proofs given are usually nonconstructive, but once you become more comfortable with constructive reasoning it will generally not be too hard to constructivise them as necessary.

\item \citet*{manuell2019thesis} is my PhD thesis. The Background chapter will be of most interest. It has some overlap with these notes, but covers certain topics, such as suplattices, in more detail. It also contains a crash course in categorical logic (though see \citet*{pittsLogic} for more details) and later gives examples of how it might be applied to the study of pointfree topology.

\item \citet*{henry2016} provides a constructive proof of the localic Gelfand duality, though probably most relevant is the helpful background material. It can be more difficult than the previous resources, but Henry works constructively throughout, and some of what is discussed cannot be found elsewhere.

\item \citet*{elephant2} is an enormous reference on topos theory, but Part C deals extensively with constructive pointfree topology. It is the most difficult reference of all, but it contains a huge amount of very useful information.
\end{itemize}

For further information on geometric logic, see \citet*{vickers2014continuity,vickers1997powerlocale}.
For a different but related approach to constructive pointfree topology see \citet*{bauer2009ASD} and \citet*{taylor2010analysis} on Abstract Stone Duality.

\bibliography{references}

\end{document}